\documentclass[11pt]{amsart}
\usepackage{amssymb}
\usepackage{amsrefs}
\usepackage{amsmath}
\usepackage{mathrsfs}
\usepackage{faktor}
\usepackage{tikz}

\usetikzlibrary{arrows}
\usetikzlibrary{shapes.geometric}
\usetikzlibrary{decorations.pathreplacing,decorations.markings}
\usetikzlibrary{automata,positioning,arrows,calc, quotes}
\usepackage[utf8]{inputenc}
\usepackage[top=1in,bottom=1in,right=1in,left=1in]{geometry}
\usepackage{hyperref}
\usepackage{enumitem}

\usepackage[capitalise]{cleveref}

\newtheorem{theorem}{Theorem}[section]
\newtheorem{lem}[theorem]{Lemma}
\newtheorem{proposition}[theorem]{Proposition}
\newtheorem{cor}[theorem]{Corollary}

\theoremstyle{definition}
\newtheorem{dfn}[theorem]{Definition}
\newtheorem{ex}[theorem]{Example}
\newtheorem{rmk}[theorem]{Remark}

\numberwithin{theorem}{section}

\newenvironment{theorem_no_number}[1][]{\begin{trivlist}
\item[\hskip \labelsep {\bfseries Theorem \def\temp{#1}\ifx\temp\empty  #1\else  #1\fi
.}] \itshape}  {\end{trivlist}}

\DeclareMathOperator{\dom}{dom}

\DeclareMathOperator{\im}{im}

\DeclareMathOperator{\Aut}{Aut}

\DeclareMathOperator{\modns}{mod}

\title{Subsets of groups with context-free preimages}

\author{Alex Levine}

\address{Department of Mathematics, Alan Turing Building, The University
of Manchester, Manchester M13 9PL, UK}

\email{alex.levine@manchester.ac.uk}



\begin{document}

\begin{abstract}
  We study subsets \(E\) of finitely generated groups where the set of all
  words over a given finite generating set that lie in \(E\) forms a
  context-free language. We call these sets \textit{recognisably context-free}.
  They are invariant of the choice of generating set and a theorem of Muller
  and Schupp fully classifies when the set \(\{1\}\) can be recognisably
  context-free. We show that every conjugacy class of a group \(G\) is
  recognisably context-free if and only if \(G\) is virtually free. We also
  show that a coset whose Schreier coset graph is quasi-transitive is
  recognisably context-free if and only if the Schreier coset graph
  is quasi-isometric to a tree.
\end{abstract}

\maketitle

\section{Introduction}
  For each finitely generated group, it is possible to define a wide variety of
  natural formal languages arising from different aspects of the group. One of
  the most widely studied is the \textit{word problem} of a group, which is the
  language of all words over a given finite generating set that represent the
  identity. Anisimov first introduced the word problem and showed that the word
  problem of a group \(G\) is a regular language if and only if \(G\) is finite
  \cite{Anisimov}. The class of groups with context-free word problem was shown
  to be the class of virtually free groups by Muller and Schupp
  \cite{muller_schupp} along with a result of Dunwoody \cite{fp_accessible}.
  Herbst also showed that a group has a one-counter word problem if and only if
  it is virtually cyclic, and Holt, Owens and Thomas showed that a group is
  virtually abelian of rank \(k\) if and only if its word problem is the
  intersection of \(k\) one-counter languages \cite{one_counter_semigroups}.
  Various attempts have also been made to classify groups with word
  problems that are poly-context-free languages \cite{brough}, multiple
  context-free languages \cites{Salvati2015,KrophollerSpriano} and the languages
  of blind \(k\)-counter automata \cite{ElderKambitesOstheimer}.

  A subset \(E\) of a finitely generated group is called \textit{recognisably
  context-free} if the language of all words representing elements of \(E\) is
  context-free. Recognisably context-free sets were introduced by Herbst
  \cite{one_counter_groups}, although Muller and Schupp had already studied them
  in the guise of context-free word problems. Asking if the
  set \(\{1\}\) is recognisably context-free is equivalent to asking if a word
  problem is context-free, and thus the Muller-Schupp Theorem fully classifies
  in which groups \(\{1\}\) is recognisably context-free. Whilst a group must be
  finitely generated in order to define recognisably context-free subsets, the
  choice of generating set does not matter.

  The complement of the word problem, called the \textit{coword problem}
  has also been widely studied, and asking if the coword problem is context-free
  is equivalent to asking if the set \(G \setminus \{1\}\) is recognisably
  context-free in \(G\). Many examples exist of non-virtually
  free groups with a context-free (but not deterministic) coword problem,
  including virtually abelian groups \cite{HoltReesRoverThomas},
  Higman-Thompson groups and Houghton groups \cite{LenhertSchweitzer}.
  There is a conjecture that a finitely generated group has a
  context-free coword problem if and only if it embeds into Thompson's
  group \(V\) \cites{LenhertThesis, BleakMatucciNeunhoffer}.

  Herbst's study of recognisably context-free sets showed that if a group \(G\)
  has the property that a subset \(R \subseteq G\) is rational if and only if
  \(R\) is recognisably context-free, then \(G\) is virtually cyclic
  \cite{one_counter_groups}. Various other lemmas were shown, including that the
  recognisably context-free sets are not affected by changing generating set,
  and are stable under multiplication by rational sets. A corollary to this is
  that in virtually free groups, rational sets are recognisably context-free.
  Herbst also studied the case when a finite set is recognisably (deterministic)
  context-free, showing that admitting a finite recognisably context-free subset,
  admitting a finite deterministic recognisably context-free subset and being
  virtually free are all equivalent \cite{Herbst92}.

  Carvalho studied recognisably context-free subsets, showing that a
  group is virtually free if and only if for all finitely generated subgroups
  \(H\) of \(G\) and all subsets \(K \subseteq H\), \(K\) is recognisably
  context-free in \(G\) if and only if \(K\) is recognisably context-free
  in \(H\) \cite{Carvalho}.

  Ceccherini-Silberstein and Woess studied when subgroups can be recognisably
  context-free (albeit using different nomenclature), and showed that a
  subgroup is recognisably context-free if and only if the corresponding
  Schreier coset graph is a context-free graph \cite{CeccWoess2012}; a condition
  dependant on the structure of the ends of the graph.

  %

  We first consider conjugacy classes. Whilst we do not fully classify all cases
  when a conjugacy class can be recognisably context-free, we do are able to
  classify the class of groups where every conjugacy class is recognisably
  context-free.

  \begin{theorem_no_number}[\ref{thm:conj}]
    Let \(G\) be a finitely generated group. Then every conjugacy class of
    \(G\) is recognisably context-free if and only if \(G\) is virtually free.
  \end{theorem_no_number}

  Our final section considers subgroups and cosets where the corresponding
  Schreier coset graph is quasi-transitive. It is not difficult to use the
  Muller-Schupp Theorem to answer the question for normal subgroups, however
  arbitrary subgroups require more work. Using a version of Stallings' Theorem
  for quasi-transitive graphs \cite{HamannLehnerMiraftabRuhmann}, we show
  the following:

  \begin{theorem_no_number}[\ref{thm:cosets}]
		Let \(G\) be a finitely generated group, \(H \leq G\) and \(g \in G\) be
    such that the Schreier coset graph of \((G, H)\) is quasi-transitive.
		Then \(Hg\) is recognisably context-free if and only if the Schreier coset
		graph of \((G, H)\) is a quasi-tree.
	\end{theorem_no_number}

  Since completing this paper, the author has been made aware of a result of
  Rodaro, released a few months earlier that proves the same result
  \cite{rodaro}, when taken together with the result of Ceccherini-Silberstein
  and Woess \cite{CeccWoess2012} that classifies when a Schreier coset graph is a context-free graph.
  Rodaro's method uses the context-free graphs introduced by Muller and Schupp
  \cite{MullerSchupp85}, whereas our method is proved using a recent
  generalisation of Stallings' Theorem \cite{HamannLehnerMiraftabRuhmann},
  avoiding context-free graphs entirely.
  %

  We begin with the preliminary knowledge required for later sections in
  Section~\ref{prelim_sec}. Section~\ref{sec:basic-properties} gives a
  collection of basic properties of recognisably free subsets. We then discuss conjugacy classes
  in Section~\ref{sec:conjugacy} and conclude with our results on
  subgroups and cosets in Section~\ref{sec:cosets}.

\section{Preliminaries}
	\label{prelim_sec}
  We introduce concepts that will be used later. Please note that functions
  will always be written to the right of their arguments.

  \subsection{Formal languages}
    A \textit{language} over an \textit{alphabet} (a finite set) \(\Sigma\) is a
    subset of the free monoid \(\Sigma^\ast\); the set of finite sequences of
    elements of \(\Sigma\), denoted \(a_1 \cdots a_n\), rather than \((a_1,
    \ldots, a_n)\). \textit{Words} over \(\Sigma\) are elements of
    \(\Sigma^\ast\). We will use \(\varepsilon\) to denote the empty word. Since
    group elements can be represented as words over a finite monoid generating
    set, to avoid confusion between group elements and abstract words, when
    writing the length of a word \(w\) we use \(|w|\); when writing the length
    of a group element \(g\) we write \(\|g\|\). To avoid similar confusion
    between equivalence as words, and as group elements, we write \(u =_G v\) if
    \(u\) and \(v\) are words representing the same element of a group \(G\) and
    \(u \equiv v\) if \(u\) and \(v\) are equivalent as words.

	\subsection{Regular languages}
    We give a very brief introduction to regular languages. We refer the reader
    to \cite[Section 2.5]{groups_langs_aut} or \cite[Chapters
    2-4]{hopcroft_motwani_ullman} for more information.

		\begin{dfn}
			Let \(\Sigma\) be an alphabet (a finite set) and let \(\Gamma\) be a
			\((\Sigma \cup \{\varepsilon\})\)-edge-labelled graph. A word \(w \in
			\Sigma^\ast\) \textit{traces a path} in \(\Gamma\) from a vertex \(u \in
			V(\Gamma)\) to \(v \in V(\Gamma)\) if there is a path \(\gamma\) in
			\(\Gamma\) from \(u\) to \(v\) such that concatenating the labels of the
			edges in \(\gamma\) (in order) yields \(w\).
		\end{dfn}

		\begin{dfn}
			A \textit{finite-state automaton} is a tuple \(\mathcal A = (\Sigma, \
			\Gamma, \ q_0, \ F)\), where
			\begin{enumerate}
				\item \(\Sigma\) is an alphabet;
				\item \(\Gamma\) is a finite edge-labelled directed graph with labels
				from \(\Sigma \cup \{\varepsilon\}\);
				\item \(q_0 \in V(\Gamma)\) is called the \textit{start state};
				\item \(F \subseteq V(\Gamma)\) is called the set of \textit{accept
				states}.
			\end{enumerate}
			We call vertices in \(\Gamma\) \textit{states}.
			A word \(w \in \Sigma^\ast\) is \textit{accepted} by \(\mathcal A\) if
			there is a path in \(\Gamma\) from \(q_0\) to a state in \(F\), where
			\(w\) is the word obtained by concatenating the labels of the edges in the
			path. The \textit{language accepted} by \(\mathcal A\) is the set of all
			words accepted by \(\mathcal A\).
			A language is called \textit{regular} if it accepted by a finite-state
			automaton.
		\end{dfn}

		\begin{ex}
			We will show that the language \(L = \{a^m b c^n \mid m, \ n \in
			\mathbb{Z}_{\geq 0}\}\) is regular over \(\{a, \ b, \ c\}\).
			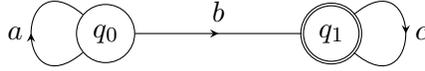
\begin{figure}
				\begin{tikzpicture}
					[scale=.6, auto=left,every node/.style={circle}]
					\tikzset{
					on each segment/.style={
						decorate,
						decoration={
							show path construction,
							moveto code={},
							lineto code={
								\path [#1]
								(\tikzinputsegmentfirst) -- (\tikzinputsegmentlast);
							},
							curveto code={
								\path [#1] (\tikzinputsegmentfirst)
								.. controls
								(\tikzinputsegmentsupporta) and (\tikzinputsegmentsupportb)
								..
								(\tikzinputsegmentlast);
							},
							closepath code={
								\path [#1]
								(\tikzinputsegmentfirst) -- (\tikzinputsegmentlast);
							},
						},
					},
					mid arrow/.style={postaction={decorate,decoration={
								markings,
								mark=at position .5 with {\arrow[#1]{stealth}}
							}}},
				}

					\node[draw] (q0) at (0, 0) {\(q_0\)};
					\node[draw, double] (q1) at (5, 0) {\(q_1\)};

					\draw[postaction={on each segment={mid arrow}}] (q1) to
					[out=40, in=-40, distance=2cm] (q1);

					\draw[postaction={on each segment={mid arrow}}] (q0) to
					[out=-140, in=140, distance=2cm] (q0);

					\draw[postaction={on each segment={mid arrow}}] (q0) to (q1);

					\node (l1) at (-2, 0) {\(a\)};

					\node (l2) at (2.5, 0.5) {\(b\)};

					\node (l3) at (7, 0) {\(c\)};

				\end{tikzpicture}
				\caption{Finite state automaton for for \(\{a^m b c^n \mid m, \ n \in
				\mathbb{Z}_{\geq 0}\}\), with start state \(q_0\) and accept state
				\(q_0\).}
				\label{reg_ex_fig}
			\end{figure}
			The finite-state automaton defined in Figure \ref{reg_ex_fig} accepts a
			language that is contained in \(L\), as reading any word in the automaton
			results in reading any number of \(a\)s, followed by one \(b\), followed
			by any number of \(c\)s. Moreover, if \(w = a^m b c^n \in L\), then we can
			use this automaton to accept \(w\) by traversing the labelled by \(a\) at
			\(q_0\) \(m\) times, then reading one \(b\) to transfer to \(q_1\), then
			traversing the \(c\) edge \(n\) times, before being accepted. Thus this
			automaton accepts \(L\), and \(L\) is a regular language.
		\end{ex}
%

	\subsection{Context-free languages}
		We define context-free languages. We give a very brief introduction to this
    class, but the reader can find more information in \cite[Section 2.6]{groups_langs_aut} or \cite[Chapters 5-7]{hopcroft_motwani_ullman}.

		\begin{dfn}
			A \textit{context-free grammar} is a tuple \(\mathcal G = (\Sigma,  V,  \mathcal P,
			 \mathbf S)\), where
			\begin{enumerate}
				\item \(\Sigma\) is a finite alphabet;
				\item \(V\) is a finite alphabet, disjoint from \(\Sigma\), called the set
				 of \textit{non-terminals};
        \item \(\mathcal P\) is a finite subset of \(V \times (\Sigma \cup
          V)^\ast\), called the set of \textit{productions}. The production
          \((\mathbf A, \
				\omega)\) is usually denoted \(\mathbf A \to \omega\).
				\item \(\mathbf S \in V\) is called the \textit{start symbol}.
			\end{enumerate}
			An \textit{application} of a production \(\mathbf A \to \omega\) to a
			word \(\nu \in (\Sigma \cup V)^\ast\) that contains \(\mathbf A\) is the
			action that replaces an occurrence of \(\mathbf A\) in \(\omega\). A word
			\(w \in \Sigma^\ast\) is \textit{generated} by \(\mathcal G\), if \(w\)
			can be obtained from \(\mathbf S\), by a finite sequence of applications
			of productions. The \textit{language generated} by \(\mathcal G\),
			denoted \(L(\mathcal G)\), is the set of all words generated by
			\(\mathcal G\). A language that is generated by a context-free grammar is
			called \textit{context-free}.

			A \textit{derivation} in \(\mathcal G\) is a finite sequence of
			applications of productions. We write \(\mathbf A \Rightarrow^\ast
			\omega\), for \(\mathbf A \in V\) and \(\omega \in (\Sigma \cup
			V)^\ast\), if there is a derivation that takes \(\mathbf A\) to
      \(\omega\). A non-terminal \(\mathbf{A}\) in \(\mathcal G\) is called
      \textit{useless} is there is no derivation in \(\mathcal G\)
      taking the start symbol \(\mathbf{S}\) to a word in the terminals
      via a word (in any combination of terminals and non-terminals)
      containing \(\mathbf{A}\).
		\end{dfn}

		\begin{ex}
			\label{ex:wp-Z-CF}
			The language
			\[
				L = \{w \in \{a, a^{-1}\}^\ast \mid w \text{ contains the
				same number of occurrences } a \text{ as } a^{-1}\}
			\]
			is context-free. We give an example of a context-free grammar for \(L\).
			Let \(\mathcal G = (\{a, a^{-1}\}, \{\mathbf{S}\}, \mathcal P, \mathbf S)\)
			be a context-free grammar, where \(\mathcal P\) contains the productions:
			\[
				\mathbf{S} \to \mathbf{S} a \mathbf S a^{-1} \mathbf{S}, \quad
				\mathbf{S} \to \mathbf S a^{-1} \mathbf S a \mathbf S, \quad
				\mathbf{S} \to \varepsilon.
			\]
			We claim that \(\mathcal G\) generates \(L\). Firstly, note that every
			word in \(L\) can be obtained from the empty word \(\varepsilon\) by a
			finite sequence of free expansions; that is, by iteratively inserting a
			subword of the form \(aa^{-1}\) or \(a^{-1}a\). By using the first two
			productions, we can therefore start with \(\mathbf{S}\) and end with every
			word in \(w \in L\) with a number of occurrences of \(\mathbf{S}\) `mixed
			in'. We can use the third production to remove all occurrences of \(\mathbf
			S\), to end up with \(w\). Conversely, any word that \(\mathcal G\)
			generates must be obtainable from \(\varepsilon\) by a finite sequence of
			free expansions, from the construction of \(\mathcal G\), and so
			\(\mathcal G\) only generates words in \(L\). Thus \(\mathcal G\)
      generates \(L\), as required.
		\end{ex}

		The following lemma collects the standard closure properties of context-free
		languages.

    \begin{lem}[{\cite[Propositions 2.6.26, 2.6.32 and 2.6.34]{groups_langs_aut}}]
			\label{lem:cf-closure}
      The class of context-free languages is closed under finite union,
      intersection with a regular languages, concatenation, Kleene star
      closure, image under free monoid homomorphism and preimage under free
      monoid homomorphism.
		\end{lem}

 		It is useful to be able to assume some context-free grammars we use are in
 		Chomsky normal form. We give the definition below.

		\begin{dfn}
			A context-free grammar \((\Sigma, \ V, \ P, \ \mathbf S)\) is in
			\textit{Chomsky normal form} if every production is of the form \(\mathbf A
			\to \mathbf{BC}\) or \(\mathbf A \to \alpha\), where \(\mathbf{A, \ B, \ C}
			\in V\) and \(\alpha \in \Sigma\).
		\end{dfn}

		\begin{lem}[{\cite[Theorem 2.6.14]{groups_langs_aut}}]
			\label{Chomsky_norm_form_lem}
			Every context-free language is accepted by a context-free grammar in
			Chomsky normal form with no useless non-terminals..
		\end{lem}

    We will also need the fact the context-free languages are closed
    under substitutions of context-free languages.

    \begin{dfn}
      Let \(L\) and \(M\) be languages over an alphabet \(\Sigma\) and let \(a
      \in \Sigma\). The \textit{substitution} of \(a\) in \(L\) for \(M\)
      is the language of all words obtained from a word in \(L\) by replacing
      each occurrence of \(a\) with a word in \(M\). That is,
      \[
        \{u_0 v_1 u_1 v_2 \cdots v_n u_n \mid u_0 a u_1 a \cdots a u_n \in L,
          u_0, \ldots, u_n \in (\Sigma \setminus \{a\})^\ast, v_1, \ldots,
        v_n \in M\}.
      \]
    \end{dfn}

    \begin{lem}[{\cite[Theorem 7.23]{hopcroft_motwani_ullman}}]
      \label{lem:cf-subs}
      Let \(L\) and \(M\) be context-free languages over an alphabet \(\Sigma\)
      and let \(a \in \Sigma\). Then the substitution of \(a\) in \(L\) for
      \(M\) is context-free.
    \end{lem}

		\subsection{Pushdown automata}
      An alternative definition for the class of context-free languages is the
      class of languages accepted by a pushdown automaton. We give the
      definition below. We refer the reader to \cite[Chapter
      6]{hopcroft_motwani_ullman} for a more detailed introduction.

      We informally describe a pushdown automaton before we give the definition.
      The idea is much the same as a finite-state automaton, with the exception
      that there is some memory - in the form of a finite word called the stack.
      When transitioning from one state to another, instead of just looking at
      what state one is currently in along with the letter (or word) being
      read, the top of the stack is also considered. When transitioning, one can
      remove (`pop') a (possibly empty) word from the top of the stack; that is
      remove a suffix. If the correct suffix does not exist in the stack, then
      the transition in question cannot be used. After popping a word, a new
      (again, possibly empty) word can be added (`pushed') to the end of the
      stack. All transitions take this form, and again the set of words that
      trace a path in the set of states, starting with the stack empty (we
      take empty stacks to contain precisely one letter, the bottom of stack
      symbol) and ending in an accept state (with any stack).

			There are multiple (equivalent) definitions of pushdown automata. Some
			have bottom of stack symbols, whilst others do not. Most standard definitions
			only allow one letter (or \(\varepsilon\)) to be read at a time. We allow
			any word to be read at a time, and thus this is what some authors call a
			\textit{generalised pushdown automaton}.

			\begin{dfn}
				A \textit{pushdown automaton} is a \(7\)-tuple \(\mathcal A = (Q, \Sigma, \chi,
				\perp, \delta, q_0, F)\), where
				\begin{enumerate}
					\item \(Q\) is a finite set, called the set of \textit{states};
					\item \(\Sigma\) is a (finite) alphabet;
					\item \(\chi\) is a finite alphabet, disjoint from \(\Sigma\), called
					the \textit{stack alphabet};
					\item \(\perp \in \chi\) is called the \textit{bottom of stack
					symbol};
					\item \(\delta \subseteq (Q \times \Sigma^\ast \times \chi^\ast)
					\times (Q \times \chi^\ast)\) is a finite set called the
					\textit{transition relation}. We must have that pairs in \(\delta\)
					can only have at most one occurrence of \(\perp\) in each tuple in the
					pair, and if it occurs in one pair, it must occur in the other. This
					is to ensure that the bottom of stack symbol always tells us when the
					stack (the `memory') is empty, and can never be removed. Transitions
					can be thought of as (not well-defined) functions, from \(Q \times
					\Sigma^\ast \times \chi^\ast\) to \(Q \times \chi^\ast\); they are not
					(necessarily) well-defined as each point in the `domain' can have
					multiple `images'.
					\item \(q_0 \in Q\) is called the start state;
					\item \(F \subseteq Q\) is called the set of \textit{accept states}.
				\end{enumerate}
				We say that \(\mathcal Q\) is \textit{deterministic} if for all
				stack words \(\nu \in \chi^\ast\), all states \(q \in Q\) and all \(w
				\in \Sigma^\ast\), there is a unique transition (or sequence of
				transitions) from \(q\) reading \(w\) for this given stack word \(\nu\).
				The language \textit{accepted} by \(\mathcal A\) is the language of all
				words \(w\) over \(\Sigma^\ast\) such that there is a finite sequence of
				transitions taking \((q_0, \perp)\) to \((q_f, \nu)\) whilst reading
				\(w\), such that \(q_f \in F\) and \(\nu \in \chi^\ast\) is any stack.
				We write \(Q(\mathcal A)\) and \(\chi(\mathcal A)\) for the set of
				states and stack alphabet, respectively of a pushdown automaton
				\(\mathcal A\).
			\end{dfn}

			\begin{ex}
				\label{ex:pda}
				We saw in Example~\ref{ex:wp-Z-CF} that the language
				\[
					L = \{w \in \{a, a^{-1}\}^\ast \mid w \text{ contains the
					same number of occurrences } a \text{ as } a^{-1}\}
				\]
				is context-free. We now define a pushdown automaton that accepts
				\(L\). This idea of the pushdown automaton is to use the stack to
				track the freely reduced form of the word read so far, and then only
				accept when the stack is empty. We formally define the automaton, but
				Figure~\ref{fig:pda-ex} contains a graphical representation. Our
				set of states will be \(\{q_0, q_1\}\), where \(q_0\) is the start
				state and \(q_1\) is the (unique) accept state. The stack alphabet
				will be \(\chi = \{\perp, x, x^{-1}\}\), with \(\perp\) the bottom of
				stack symbol. We then have six transitions from \(q_0\) to \(q_0\) and
				one transition from \(q_0\) to \(q_1\):
				\begin{enumerate}
					\item \((q_0, a, \perp) \to (q_0, x \perp)\);
					\item \((q_0, a, x) \to (q_0, xx)\);
					\item \((q_0, a, x^{-1}) \to (q_0, \varepsilon)\);
					\item \((q_0, a^{-1}, \perp) \to (q_0, x^{-1} \perp)\);
					\item \((q_0, a^{-1}, x) \to (q_0, \varepsilon)\);
					\item \((q_0, a^{-1}, x^{-1}) \to (q_0, x^{-1} x^{-1})\);
					\item \((q_0, \varepsilon, \perp) \to (q_1, \perp)\);
				\end{enumerate}
				The first six transitions simply track the freely reduced form of the
				word read so far (except using \(x\) rather than \(a\)) and the last
				transition confirms that the stack is empty; that is, that the word
				indeed equals the identity of the group \(\langle a \mid \rangle\) (that
				is, it lies in \(L\)) before moving to the accept state. If we move to
				\(q_1\) before finishing reading our word, we can never accept, as there
				are no transitions out of \(q_1\) that allow the rest of the word to be
				read.
				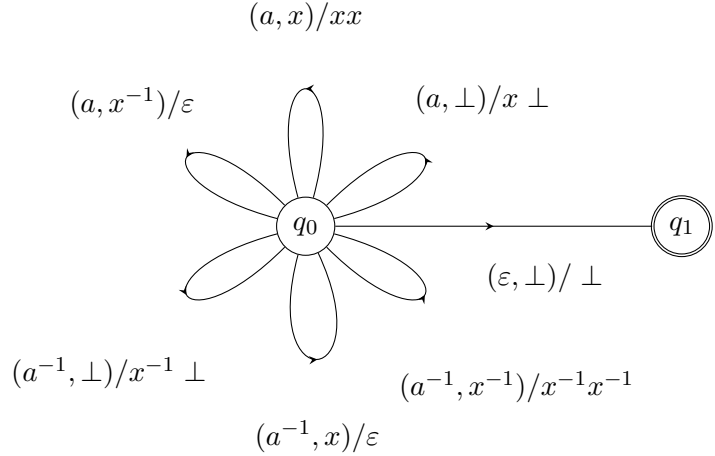
\begin{figure}
					\begin{tikzpicture}
						[scale=1, auto=left,every node/.style={circle}]
						\tikzset{
						on each segment/.style={
							decorate,
							decoration={
								show path construction,
								moveto code={},
								lineto code={
									\path [#1]
									(\tikzinputsegmentfirst) -- (\tikzinputsegmentlast);
								},
								curveto code={
									\path [#1] (\tikzinputsegmentfirst)
									.. controls
									(\tikzinputsegmentsupporta) and (\tikzinputsegmentsupportb)
									..
									(\tikzinputsegmentlast);
								},
								closepath code={
									\path [#1]
									(\tikzinputsegmentfirst) -- (\tikzinputsegmentlast);
								},
							},
						},
						mid arrow/.style={postaction={decorate,decoration={
									markings,
									mark=at position .5 with {\arrow[#1]{stealth}}
								}}},
					}

						\node[draw] (q0) at (0, 0) {\(q_0\)};
						\node[draw, double] (q1) at (5, 0) {\(q_1\)};

						\draw[postaction={on each segment={mid arrow}}] (q0) to
						[out=15, in=45, distance=2cm] node[midway, above right]{\((a, \perp) / x \perp\)} (q0);
						\draw[postaction={on each segment={mid arrow}}] (q0) to
						[out=75, in=105, distance=2cm] node[midway, above]{\((a, x) / xx\)} (q0);
						\draw[postaction={on each segment={mid arrow}}] (q0) to
						[out=135, in=165, distance=2cm] node[midway, above left]{\((a, x^{-1}) / \varepsilon\)} (q0);
						\draw[postaction={on each segment={mid arrow}}] (q0) to
						[out=195, in=225, distance=2cm] node[midway, below left]{\((a^{-1}, \perp) / x^{-1} \perp\)} (q0);
						\draw[postaction={on each segment={mid arrow}}] (q0) to
						[out=255, in=295, distance=2cm] node[midway, below]{\((a^{-1}, x) / \varepsilon\)} (q0);
						\draw[postaction={on each segment={mid arrow}}] (q0) to
						[out=315, in=345, distance=2cm] node[midway, below right]{\((a^{-1}, x^{-1}) / x^{-1} x^{-1}\)} (q0);

						\draw[postaction={on each segment={mid arrow}}] (q0) to
						node[midway, below right]{\((\varepsilon, \perp) / \perp\)} (q1);

					\end{tikzpicture}
					\caption{Pushdown automaton defined in Example~\ref{ex:pda} that accepts \(L = \{w \in \{a, a^{-1}\}^\ast
					\mid w \text{ contains the same number of occurrences } a \text{ as }
					a^{-1}\}\). The start state is \(q_0\) and the accept state is
					\(q_1\). Each transition from a state to a state is written in the
					form \((b, \alpha) / \beta\), where \(b\) is the (terminal) letter
					read, \(\alpha\) is the stack word popped from the top of the stack
					and \(\beta\) is the stack word pushed to the top of the stack.}
					\label{fig:pda-ex}
				\end{figure}
			\end{ex}

      \begin{lem}[{\cite[Theorem 2.6.10]{groups_langs_aut}}]
				A language is context-free if and only if it is accepted by a pushdown
				automaton.
			\end{lem}

			\begin{dfn}
				A language is called \textit{deterministic context-free} if it is
				accepted by a deterministic pushdown automaton.
			\end{dfn}

      \begin{lem}[{\cite[Propositions 2.6.30 and 2.6.34]{groups_langs_aut}}]
				\label{lem:det-cf-closure}
				The class of deterministic context-free languages is closed under
        complement and preimage under free monoid homomorphism.
			\end{lem}

			We will need the following lemma when classifying recognisably context-free
			cosets.

			\begin{lem}
				\label{lem:pda-empty-stack}
				Let \(L\) be a context-free language. Then \(L\) is accepted by a
				pushdown automaton \(\mathcal A\), such that whenever an accept state
				in \(\mathcal A\) is reached, the stack is always empty.
			\end{lem}

			\begin{proof}
				As \(L\) is context-free, there is a pushdown automaton \(\mathcal A\)
				accepting \(L\), with set of accept states \(F\), stack alphabet
				\(\chi\) and bottom of stack symbol \(\perp\). We modify \(\mathcal A\)
				to obtain a new pushdown automaton \(\mathcal B\) as follows. We start
				by adding two new states \(q_1\) and \(q_2\) to \(\mathcal A\), and
				redefine the set of accept states to be \(\{q_2\}\). We then add an
				\(\varepsilon\)-transition from each \(q \in F\) to \(q_1\) that does
				not alter the stack. For each \(\alpha \in \chi \setminus \{\perp\}\) we
				add an \(\varepsilon\)-transition from \(q_1\) to \(q_1\) that pops
				\(\alpha\) from the stack. We then add a \(\varepsilon\)-transition from
				\(q_1\) to \(q_2\) that pops \(\perp\) from the stack and then pushes
				\(\perp\) back onto the stack. By construction, \(\mathcal B\) can only
				accept when the stack is empty and \(\mathcal B\) must accept the same
				language as \(\mathcal A\).
			\end{proof}

		\subsection{Recognisable and rational sets}

			Before we formally define recognisably context-free sets, we first cover
			their regular analogues: recognisable sets. We give the definition below.

			\begin{dfn}
				Let \(G\) be a group with a finite generating set \(\Sigma\) and let
				\(\pi \colon \Sigma^\ast \to G\) be the natural homomorphism. A subset
				\(E \subseteq G\) is called \textit{recognisable} with respect to
				\(\Sigma\) if \(E \pi^{-1}\) is a regular language.
			\end{dfn}

			Using the same argument as the proof of Lemma~\ref{gen_set_lem}, we have
			that changing finite generating sets in a finitely does not affect whether
			a given subset is recognisable.
			Herbst and Thomas completely characterised recognisable subsets of groups
			in the following result.

			\begin{proposition}[{\cite[Proposition 6.3]{HerbstThomas}}]
				\label{prop:rec-classification}
				A subset \(E\) of a finitely generated group \(G\) is recognisable
				if and only if \(E\) is a finite union of cosets of some finite-index
				subgroup of \(G\).
			\end{proposition}

			A dual notion to recognisable sets is the concept of rational sets.
			Rather than having a regular full preimage, these are the image of a
			regular language under the natural map \(\pi\), or equivalently.

			\begin{dfn}
				Let \(G\) be a group with a finite generating set \(\Sigma\), and let
				\(\pi \colon \Sigma^\ast \to G\) be the natural map. A subset
				\(E \subseteq G\) is called \textit{rational} of \(E = L \pi\) for some
				regular language \(L \subseteq \Sigma^\ast\).
			\end{dfn}

			As with recognisable sets, a similar argument to that in the proof of
			Lemma~\ref{gen_set_lem} shows that the class of rational sets is
			invariant under changing finite generating set.

		\subsection{Recognisably context-free sets}
			The earliest reference we can find to recognisably context-free sets is in
			a paper of Herbst \cite{one_counter_groups}. While they are not given a
			name, the set of all recognisably context-free subsets of a group is denoted
			\(\text{CF}(G)\). Carvalho uses the term \textit{context-free} instead of
			recognisably context-free \cite{Carvalho}; we avoid this to maintain a
			clear distinction between recognisably context-free subsets of groups and
			context-free languages. We instead use the term recognisably context-free
			from \cite{CiobanuEvettsLevine}. We begin with the definition.

      \begin{dfn}
        Let \(G\) be a finitely generated group, \(\Sigma\) be a finite monoid
        generating set, and \(\pi \colon \Sigma^\ast \to G\) be the natural
        homomorphism. A subset \(E \subseteq G\) is called \textit{recognisably
        (deterministic) context-free} with respect to \(\Sigma\) if the full
        preimage \(E \pi^{-1}\) is (deterministic) context-free.
      \end{dfn}

			The following lemma is well-known (see, for example \cite[Lemma
			2.1]{one_counter_groups}). We include a short proof for completeness.

			\begin{lem}
				\label{gen_set_lem}
				Let \(G\) be a finitely generated group. If \(E \subseteq G\) is
				recognisably (deterministic) context-free with respect to one finite
				monoid generating set of \(G\), then \(E\) is recognisably
				(deterministic) context-free with respect to all finite monoid
				generating sets of \(G\).
			\end{lem}

			\begin{proof}
				Let \(\Sigma\) and \(\Delta\) be finite monoid generating sets for
				\(G\), and \(\pi_\Sigma \colon \Sigma^\ast \to G\) and \(\pi_\Delta
				\colon  \Delta^\ast \to G\) be the natural homomorphisms. Suppose that
				\(E\) is (deterministic) recognisably context-free with respect to
				\(\Sigma\). For all \(a \in \Delta\) there exists \(\omega_a \in
				\Sigma^\ast\), such that \(\omega_a =_G a\). Define the free monoid
				homomorphism:
				\begin{align*}
					\phi \colon \Delta^\ast & \to \Sigma^\ast \\
					a & \mapsto \omega_a.
				\end{align*}
        Then \(E \pi_\Delta^{-1} = (E \pi_\Sigma^{-1}) \phi^{-1}\), which is
        context-free as the class of (deterministic) context-free languages is
        closed under preimages of free monoid homomorphisms
        (Lemma~\ref{lem:cf-closure} and Lemma~\ref{lem:det-cf-closure}).
			\end{proof}

      As Lemma~\ref{gen_set_lem} shows, whether or not a subset of a group
      is recognisably context-free is not affected by the choice of
      generating set. Thus we can say that a subset of a group is
      recognisably context-free, omitting the generating set. As the above
      proof only relies on the fact that the class of languages is closed
      under preimages of free monoid homomorphisms, it holds for recognisable
      (regular) sets as well, so we will also omit the generating set when
      referring to such sets.

		\subsection{Quasi-isometries, trees and triangulations}
			Quasi-isometries between metric spaces are a central notion to geometric
			group theory. In the later sections we will show certain groups are
			virtually free is by showing that their Cayley graphs are quasi-isometric
			to trees. We give a brief definition along with a characterisation of
			graphs that are quasi-isometric to trees. We refer the reader to
			\cite[Section 11]{Meier_groups_graphs_trees} for an in-depth introduction
			to quasi-isometries from the viewpoint of geometric group theory.

	      \begin{dfn}
	        Let \(X\) and \(Y\) be metric spaces. A function \(f \colon X \to Y\) is
	        a \textit{quasi-isometry} if there exist constants \(\lambda \geq 1\) and
	        \(\mu \geq 0\), such that:
	        \begin{enumerate}
	  	\item For all \(x_1, \ x_2 \in X\), \(\frac{1}{\lambda}d(x_1, \ x_2) -
	  		\mu \leq d((x_1)f, \ (x_2)f) \leq \lambda d(x_1, \ x_2) +
	  		\mu\);
	  	\item For all \(y \in Y\) there exists \(x \in X\), such that
	  		\(d((x)f, \ y) \leq \mu\).
	        \end{enumerate}
	        If a quasi-isometry from \(X\) to \(Y\) exists, we say \(X\) and \(Y\) are
	        \textit{quasi-isometric}.
	      \end{dfn}

	      \begin{rmk}
          The property of being quasi-isometric is symmetric, reflexive and
          transitive.
	      \end{rmk}

				A particular class of graphs we will be using frequently is the class of
				graphs that are quasi-isometric to trees.

				\begin{dfn}
					A graph is called a \textit{quasi-tree} if it is quasi-isometric to a
					tree.
				\end{dfn}

        As most of the graphs we deal with will be locally finite, we define
        this concept as well.

        \begin{dfn}
          A graph is called \textit{locally finite} if the degree of every
          vertex is finite.
        \end{dfn}

	      We now define a triangulation. The definition of a triangulation is the
	      one used in \cite{antolin_Cayley_VF}. Triangulations are a key part in
	      the proof of the result of Muller and Schupp. Showing that there exists
	      \(m \in \mathbb{Z}_{> 0}\) such that every circuit in a given graph is
	      \(m\)-triangulable is sufficient to show that this graph is a
	      quasi-tree. This is often easier than explicitly constructing a
	      quasi-isometry.

	      \begin{dfn}
	        Let \(\Gamma\) be a graph. Let \(m, \ n \in \mathbb{Z}_{> 0}\). An
	        \textit{\(m\)-sequence} of length \(n\) in \(\Gamma\) is a sequence
	        \((v_0, \ \ldots, \ v_n)\) of elements of \(V(\Gamma)\), such that \(v_0
	        = v_n\), and \(d_\Gamma(v_i, \ v_{i + 1}) \leq m\) for all \(i\). An
	        \(m\)-sequence is called \textit{\(m\)-reducible}, if there exists \(i
	        \in \{1, \ \ldots, \ n - 1\}\), such that \(d_\Gamma(v_{i - 1}, \ v_{i +
	        1}) \leq m\). In such a case, an \textit{\(m\)-reduction} of this
	        \(m\)-sequence at \(i\), is the operation that outputs the \(m\)-sequence
	        \((v_0, \ \ldots, \ v_{i - 1}, \ v_{i + 1}, \ \ldots, \ v_n)\).

	        An \textit{\(m\)-triangulation} of an \(m\)-sequence is a finite sequence of
	        \(m\)-reductions that results in an \(m\)-sequence of length at most \(4\)
	        (with \(3\) distinct points, as the first and last are equal), called the
	        \textit{core} of the triangulation. An \(m\)-sequence that admits an
	        \(m\)-triangulation is called \(m\)-triangulable. Note that as
	        \(1\)-sequences are \(m\)-sequences, we can say that a \(1\)-sequence is
	        \(m\)-triangulable.
	    \end{dfn}

	    \begin{rmk}
	      The \(m\)-reductions in an \(m\)-triangulations can be depicted by drawing
	      a line. This lets us depict the entire triangulation as a number of lines
	      added to our circuit (see \cref{triangulation_fig}).
	    \end{rmk}

	        \begin{figure}
	      \begin{tikzpicture}
	        [scale=1, auto=left,every node/.style={circle, fill=black, scale=0.3}]

	        \node[label={right:{\scalebox{3}{\(v_1\)}}}] (v1) at (0:1) {};
	        \node[label={right:{\scalebox{3}{\(v_2\)}}}] (v2) at (30:1) {};
	        \node[label={right:{\scalebox{3}{\(v_3\)}}}] (v3) at (60:1) {};
	        \node (v4) at (90:1) {};
	        \node (v5) at (120:1) {};
	        \node (v6) at (150:1) {};
	        \node (v7) at (180:1) {};
	        \node (v8) at (210:1) {};
	        \node (v9) at (240:1) {};
	        \node (v10) at (270:1) {};
	        \node (v11) at (300:1) {};
	        \node (v12) at (330:1) {};

	        \draw (v1) to (v2);
	        \draw (v2) to (v3);
	        \draw (v3) to (v4);
	        \draw (v4) to (v5);
	        \draw (v5) to (v6);
	        \draw (v6) to (v7);
	        \draw (v7) to (v8);
	        \draw (v8) to (v9);
	        \draw (v9) to (v10);
	        \draw (v10) to (v11);
	        \draw (v11) to (v12);
	        \draw (v12) to (v1);
	      \end{tikzpicture}
	       \begin{tikzpicture}
	        \node (0) at (0, 0) {};
	        \node (1) at (0, 1.15) {\(\to\)};
	        \node (2) at (0.5, 0) {};
	      \end{tikzpicture}
	      \begin{tikzpicture}
	        [scale=1, auto=left,every node/.style={circle, fill=black, scale=0.3}]

	        \node[label={right:{\scalebox{3}{\(v_1\)}}}] (v1) at (0:1) {};
	        \node[label={right:{\scalebox{3}{\(v_2\)}}}] (v2) at (30:1) {};
	        \node[label={right:{\scalebox{3}{\(v_3\)}}}] (v3) at (60:1) {};
	        \node (v4) at (90:1) {};
	        \node (v5) at (120:1) {};
	        \node (v6) at (150:1) {};
	        \node (v7) at (180:1) {};
	        \node (v8) at (210:1) {};
	        \node (v9) at (240:1) {};
	        \node (v10) at (270:1) {};
	        \node (v11) at (300:1) {};
	        \node (v12) at (330:1) {};

	        \draw (v1) to (v2);
	        \draw (v2) to (v3);
	        \draw (v3) to (v4);
	        \draw (v4) to (v5);
	        \draw (v5) to (v6);
	        \draw (v6) to (v7);
	        \draw (v7) to (v8);
	        \draw (v8) to (v9);
	        \draw (v9) to (v10);
	        \draw (v10) to (v11);
	        \draw (v11) to (v12);
	        \draw (v12) to (v1);

	        \draw (v1) to (v3);
	      \end{tikzpicture}
	      \\
	       \begin{tikzpicture}
	        [scale=1, auto=left,every node/.style={circle, fill=black, scale=0.3}]

	        \node (v1) at (0:1) {};
	        \node (v2) at (30:1) {};
	        \node (v3) at (60:1) {};
	        \node (v4) at (90:1) {};
	        \node (v5) at (120:1) {};
	        \node (v6) at (150:1) {};
	        \node (v7) at (180:1) {};
	        \node (v8) at (210:1) {};
	        \node (v9) at (240:1) {};
	        \node (v10) at (270:1) {};
	        \node (v11) at (300:1) {};
	        \node (v12) at (330:1) {};

	        \draw (v1) to (v2);
	        \draw (v2) to (v3);
	        \draw (v3) to (v4);
	        \draw (v4) to (v5);
	        \draw (v5) to (v6);
	        \draw (v6) to (v7);
	        \draw (v7) to (v8);
	        \draw (v8) to (v9);
	        \draw (v9) to (v10);
	        \draw (v10) to (v11);
	        \draw (v11) to (v12);
	        \draw (v12) to (v1);

	        \draw (v1) to (v3);
	        \draw (v1) to (v11);
	        \draw (v4) to (v6);
	        \draw (v6) to (v8);
	        \draw (v6) to (v9);
	        \draw (v6) to (v10);
	        \draw (v3) to (v6);
	        \draw (v10) to (v1);
	        \draw (v1) to (v6);

	      \end{tikzpicture}
	        \caption{Triangulations}
	        \label{triangulation_fig}
	    \end{figure}
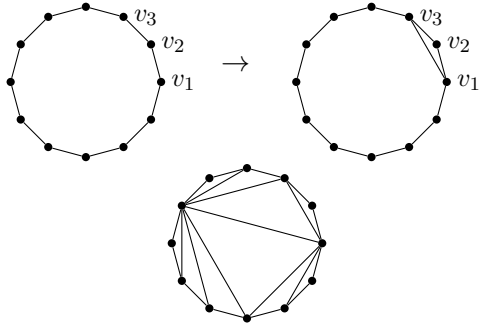

	    Part of the proof of the Muller-Schupp Theorem involves showing that
	    the Cayley graphs in groups with a context-free word problem are
	    \(m\)-triangulable for some \(m \in \mathbb{Z}_{> 0}\). It is well-known
	    that this property is equivalent to being quasi-isometric to a tree:

	    \begin{theorem}[{\cite[Theorem 4.7]{antolin_Cayley_VF}}]
				\label{thm:qt-triangulation}
	      A graph is \(m\)-triangulable for some \(m \in \mathbb{Z}_{> 0}\) if
	      and only if it is a quasi-tree.
	    \end{theorem}

	    After this, Stalling's Theorem, together with Dunwoody's accessibility result
	    \cite{fp_accessible} can be used to show that groups whose Cayley graphs are
	    quasi-isometric to trees are virtually free.

		\subsection{Cayley graphs and Schreier coset graphs}
			We briefly recall the definitions of Cayley graphs and Schreier coset
			graphs.
			\begin{dfn}
				Let \(G\) be a group with a finite inverse closed generating set
				\(\Sigma\). The \textit{(right) Cayley graph} of \(G\) with respect to
				\(\Sigma\) is the directed \(\Sigma\)-edge-labelled graph whose vertices
				are the elements of \(G\), and with an edge labelled \(a\) from \(g\) to
				\(ga\) for all \(g \in G\) and \(a \in \Sigma\).
			\end{dfn}

			\begin{theorem}[{\cite[Theorem 4.7]{antolin_Cayley_VF}}]
				\label{thm:cayley-qi-tree}
				A finitely generated group is virtually free if and only if it has a
        Cayley graph that is a quasi-tree.
			\end{theorem}

			\begin{dfn}
				Let \(G\) be a group with a finite inverse closed generating set
				\(\Sigma\) and let \(H \leq G\). The \textit{(right) Schreier coset
				graph} of \((G, H)\) with respect to \(\Sigma\) is the directed
				\(\Sigma\)-edge-labelled graph whose vertices are the right cosets of
				\(H\) in \(G\), and with an edge labelled \(a\) from each coset \(Hg\)
				to \(Hga\) for all \(a \in \Sigma\).
			\end{dfn}

			\begin{rmk}
				As with Cayley graphs, Schreier coset graphs are dependent on the choice
				of generating set, however different Schreier coset graphs for the same
				pair \((G, H)\), where \(G\) is finitely generated will be
				quasi-isometric. We will thus refer to the Schreier coset graph of
				\((G, H)\) when talking about properties of graphs that are invariant
				under quasi-isometries.
			\end{rmk}

			\begin{dfn}
				A graph \(\Gamma\) is called \textit{(vertex) transitive} if for all
				\(u, v \in V(\Gamma)\) there exists \(\phi \in \Aut(\Gamma)\) with \((u)
				\phi = v\); that is, \(\Gamma\) has a unique automorphic orbit. We say
				\(\Gamma\) is \textit{quasi-transitive} if it has finitely many
				automorphic orbits.
			\end{dfn}

		\subsection{Tree amalgamations}
      \label{subsec:tree-amalg}
		  The Muller-Schupp Theorem was proved by first showing that if a finitely
		  generated group has a  context-free word problem, then its Cayley graph is
		  quasi-isometric to a tree. Then Stallings' Theorem together with
		  Dunwoody's (later) accessibility result showed that a finitely generated
		  group that is quasi-isometric to a tree is virtually free. At this point,
			the proof can be completed by showing that virtually free groups have
			(deterministic) context-free word problems.

      In Section~\ref{sec:cosets}, we extend this result to show that a coset
      of a subgroup of a finitely generated group whose Schreier coset graph is
      quasi-transitive, is recognisably context-free if and only if the
      Schreier coset graph is a quasi-tree. For this, we a generalisation of
      Stallings' Theorem to quasi-transitive graphs. Thus gives us an
      alternative characterisation of quasi-trees which can be used to show
      that the `language' of a transitive quasi-tree is context-free.

      The definitions in the subsection are from Section 5 of
      \cite{HamannLehnerMiraftabRuhmann}. We begin with the definition of a
      tree amalgamation. All of the graphs used here are considered to be
      simple graphs (that is, no multiple edges, loops or directions). Since we
      only use tree amalgamations to show graphs are quasi-isometric, and since
      forgetting directions of Cayley and Schreier graphs of groups does not
      affect the metric, this is sufficient for our purposes. Since we have at
      most one edge between two vertices, we can define edges to be subsets of
      the set of vertices of size \(2\); the two vertices in each edge being
      its endpoints.
			\begin{dfn}
				Let \(\Gamma_1\) and \(\Gamma_2\) be graphs. Let \((S_k^i)_{i \in I_i}\)
				be a collection of subsets of \(V(\Gamma_i)\), for \(i \in \{1, 2\}\),
				where each \(I_i\) is an index set, such that all \(S_k^i\) have the
				same cardinality and \(I_1 \cap I_2 = \varnothing\). For each \(k \in
				I_1\) and \(l \in I_2\) let \(\phi_{kl} \colon S_k^1 \to S_l^2\) be a
				bijection, and let \(\phi_{lk} = \phi_{kl}^{-1}\).

				Let \(T\) be the \((|I_1|, |I_2|)\)-semiregular tree; that is the
				bipartite tree whose vertices are partitioned into \(V(T) = V_1 \cup
				V_2\) such that all vertices in \(V_i\) have degree \(|I_i|\). Let
				\(D(T)\) be the set of directed edges obtained from \(T\) by taking
				each edge \(\{u, v\}\) and taking its two directed versions \((u, v)\)
				and \((v, u)\). We also attach a labelling \(f \colon D(T) \to I_1 \cup
				I_2\) such that for all \(t \in V_i\) the set of labels of the incident
				is precisely the set \(I_i\) and each label occurs on precisely one
				incident edge.

				For each \(t \in V_i\) let \(\Gamma_t\) be an isomorphic copy of
				\(\Gamma_i\). Let \(S_k^t\) denote the copy of \(S_k^i\) within
				\(\Gamma_t\). Let \(\Lambda\) be the disjoint union the graphs \(\Gamma_t\) for all
				\(t \in V(T)\). We now quotient \(\Lambda\) as follows. For each edge \((s, t) \in D(T)\) with \((s, t))f = k\) and \((t, s)f = l\), identify all vertices \(v \in
				S_k^s\) with the vertex \((v)\phi_{kl}\) in \(S_l^{t}\). The quotient
				graph obtained is called the \textit{tree amalgamation} of \(\Gamma_1\)
				with \(\Gamma_2\) over the \textit{connecting tree} \(T\), and denoted
				\(\Gamma_1 \ast \Gamma_2\).

				These functions \(\phi_{kl}\)	called the \textit{bonding maps} of the
				tree amalgamation and the sets \(S_k^i\) are called the \textit{adhesion
				sets}. If all adhesion sets within a tree amalgamation are finite, the
				tree amalgamation is said to have \textit{finite adhesion}. The
				\textit{identification size} of a vertex \(v \in V(\Gamma_1 \ast
				\Gamma_2)\) is number of vertices in \(V(\Lambda)\) that are identified
				when quotienting to obtain \(\Gamma_1 \ast \Gamma_2\). The tree
				amalgamation has \textit{finite identification} if all identification
				sizes are finite.
			\end{dfn}

      We refer the reader to \cite[Examples
      5.2-5.6]{HamannLehnerMiraftabRuhmann} for a variety of examples of tree
      amalgamations.

      As with groups we also require an accessibility result; it is not enough
      to say graphs can successively be expressed as tree amalgamations, this
      process must terminate, and the resultant graphs must be sufficiently
      understood. We combine these concepts in the following definition.

      \begin{dfn}
        A graph \(\Gamma\) is said to admit a \textit{terminal factorisation
        of finite graphs} if there exists a finite collection of finite
        graphs from which \(\Gamma\) can be built by a finite sequence
        of successive tree amalgamations with finite adhesion and finite
        identification.
      \end{dfn}

      We state two results which together give a classification of connected
      quasi-transitive locally finite quasi-trees. These theorems mention the
      property of having only thin ends. When combined together, this property
      is not present, so we do not define it here.

      \begin{theorem}[{\cite[Theorem 5.5]{KrohnMoller}}]
        \label{thm:qt-thin-end}
        A connected quasi-transitive locally finite graph has only thin ends
        if and only if it is a quasi-tree.
      \end{theorem}

      \begin{theorem}[{\cite[Theorem 7.5]{HamannLehnerMiraftabRuhmann}}]
        \label{thm:thin-ends-fact}
        A connected quasi-transitive locally finite graph has only thin ends
        if and only if it admits a terminal factorisation of finite graphs.
      \end{theorem}

      Combining Theorem~\ref{thm:qt-thin-end} with
      Theorem~\ref{thm:thin-ends-fact} gives the following.

      \begin{theorem}
        \label{thm:qt-fact}
        A connected quasi-transitive locally finite graph is a quasi-tree
        if and only if it admits a terminal factorisation of finite graphs.
      \end{theorem}

\section{Basic properties}
  \label{sec:basic-properties}
	This section covers various basic (closure) properties of recognisably
	context-free sets. We begin with a result of Herbst.

  \begin{proposition}[\cite{one_counter_groups}, Lemma 4.1]
		\label{prop:rat-times-RCF}
  	Let \(G\) be a finitely generated group, let \(A \subseteq G\) be
		recognisably context-free and let \(R \subseteq G\) be rational. Then
		\(AR\) and \(RA\) are recognisably context-free.
  \end{proposition}

	We rarely use Proposition~\ref{prop:rat-times-RCF} in its full generality. It
	is mostly used when \(R\) is a singleton (which is always rational as the
	image of a one-word language). We thus state it in this restricted form to
	make it clear which rational subset we are using.

  \begin{cor}
    \label{cor:post-multiply-cf}
    Let \(G\) be a finitely generated group, and let \(A \subseteq G\) be
    recognisably context-free. For all \(g \in G\), \(Ag\) and \(gA\) are
    recognisably context-free.
  \end{cor}

	An interesting corollary to Proposition~\ref{prop:rat-times-RCF} is that
	rational subsets of virtually free groups are recognisably context-free.
	The converse is not true - recognisably context-free subsets of virtually
	free groups are not always rational. Conjugacy classes provide one such
	counter-example (see for example Theorem~\ref{thm:conj}). In fact, Herbst
	showed that the class of groups such that a subset is rational if and only if
	it is recognisably context-free is precisely the class of virtually cyclic
	groups \cite[Theorem 3.1]{one_counter_groups}.

  \begin{theorem}[\cite{one_counter_groups}, Lemma 4.2 and Theorem 3.1]
    In a finitely generated virtually free group, every rational subset is
    recognisably context-free. A finitely generated group \(G\) has the property
    that the classes of rational and recognisably context-free sets coincide if
    and only if \(G\) is virtually cyclic.
  \end{theorem}

	Using the fact that context-free languages are stable under intersections
	with regular languages, we can make the following observation.

  \begin{lem}
    \label{lem:cf-rec-int-reg-rec}
    The intersection of a recognisably context-free set with a recognisable
    set is recognisably context-free.
  \end{lem}

  \begin{proof}
    Let \(G\) be a finitely generated group, \(C \subseteq G\) be recognisable
    context-free, and \(R \subseteq G\) be recognisable. Fix a finite monoid
    generating set \(\Sigma\) for \(G\), and let \(\pi \colon \Sigma^\ast \to
    G\) be the natural homomorphism. Then \(C \pi^{-1} \cap R \pi^{-1} = (C \cap
    R) \pi^{-1}\). As the intersection of a context-free language with a regular
    language, this language is context-free (Lemma~\ref{lem:cf-closure}). Thus
    \(C \cap R\) is recognisably context-free.
  \end{proof}

	Using Lemma~\ref{lem:cf-rec-int-reg-rec} along with the fact that cosets of
	finite-index subgroups are always recognisable, we can classify recognisably
	context-free subsets of a group in terms of recognisably context-free subsets
	of any of its finite-index subgroups.

  \begin{proposition}[{\cite[Proposition 3.6]{Carvalho}}]
    Let \(G\) be a finitely generated group, and \(H\) be a finite-index
    subgroup. Let \(T\) be a (finite) right transversal for \(H\) in \(G\).
    Suppose \(C\) is a recognisably context-free subset of \(G\). Then for each
    \(t \in T\) there exists a recognisably context-free \(C_t\) of \(H\),
    such that
    \[
      C = \bigcup_{t \in T} C_t t.
    \]
  \end{proposition}

    %

	It is well-known that the class of context-free languages is stable under
	preimages of free monoid homomorphisms. We prove an analogous statement holds
	for recognisably context-free subsets of a given group.

  \begin{proposition}
    \label{prop:rcf-preimage}
    Let \(G\) and \(H\) be a finitely generated groups and \(\phi \colon G \to
    H\) be a epimorphism. If \(A \subseteq H\) is a recognisably context-free
    subset of \(H\), then \(A \phi^{-1}\) is a recognisably context-free subset
    of \(G\).
  \end{proposition}

  \begin{proof}
    Fix a finite monoid generating set \(\Sigma\) for \(G\). Then \(\Sigma
    \phi\) is a finite monoid generating set for \(H\). Let \(\pi_G \colon
    \Sigma^\ast \to G\) and \(\pi_H \colon (\Sigma \phi)^\ast \to H\) be the
    natural homomorphisms. Let \(\bar{\phi} \colon \Sigma^\ast \to (\Sigma
    \phi)^\ast\) be the homomorphism that extends \(a \mapsto a \phi\) for all
    \(a \in \Sigma\). We will show that \(A \phi^{-1} \pi_G^{-1} = A \pi_H^{-1}
    \bar{\phi}^{-1}\).  To show this, we need to show that \(w \bar{\phi} \pi_H
    \in A\) if and only if \(w \pi_G \phi \in A\). But this is true, as the
    element of \(H\) that \(w \bar{\phi}\) represents is \(w \pi_G \phi\).
  \end{proof}

	\section{Conjugacy classes}
    \label{sec:conjugacy}
		The aim of this section is to classify all finitely generated groups where
		every conjugacy class is recognisably context-free, which ends up being the
		class of virtually free groups. We do not provide a full classification of
		when conjugacy classes are recognisably context-free, although we briefly
    discuss the case when a group admits a recognisably context-free
    conjugacy class.

    Most of the work in this section is therefore to prove that conjugacy
    classes in finitely generated virtually free groups are recognisably
    context-free. Since virtually free groups always admit finite-index normal
    free subgroups, we can define the multiplication in a finitely generated
    virtually free group using a finite-index free normal subgroup, a (finite)
    right transversal, and the action (by automorphisms) of the transversal on
    the normal subgroup. We begin with the definition of a \(\phi\)-cyclic
    permutation, where \(\phi\) is an automorphism of a free group; a
    generalisation of cyclic permutations. Thus before we can show that
    conjugacy classes in virtually free groups are recognisably context-free, we
    must first show that \(\phi\)-twisted conjugacy classes are recognisably
    context-free in free groups, where \(\phi\) is a virtually inner
    automorphism. We start by defining virtually inner automorphisms.

    \begin{dfn}
      Let \(F\) be a finite rank free group. We say an automorphism
      \(\phi \in \Aut(F)\) is \textit{virtually inner} if there exists
      \(k \in \mathbb{Z}_{\geq 0}\) such that \(\phi^k\) is an inner
      automorphism.
    \end{dfn}

		\begin{dfn}
			Let \(F\) be a finite rank free group with basis \(\Sigma\), and \(\phi
			\in \Aut(F)\). Define \(\sim_\phi\) on the set of freely reduced words in
			\((\Sigma \cup \Sigma^{-1})^\ast\) to be the transitive closure of the
			binary relation
			\[
				\{(uv, \ v(u\phi)) \mid u, v \in \Sigma^\ast \text{ freely reduced}\}
				\cup \{(v(u\phi)), \ uv) \mid u, v \in \Sigma^\ast \text{ freely reduced}\},
			\]
      where \(uv\) is the freely reduced word obtained by concatenating and
      freely reducing, and \(v (u\phi)\) is the freely reduced word obtained by
      applying \(\phi\), concatenating and freely reducing.

      As \(\sim_\phi\) is defined on the set of freely reduced words, we can
      therefore define \(\sim_\phi\) on \(F\) as well. We say \(g\) is a
      \textit{\(\phi\)-cyclic permutation} if \(g \sim_\phi h\). We say \(g\) is
      a \textit{\(\phi\)-twisted conjugate} of \(h\) if there exists \(x \in F\)
      such that \(x^{-1} g (x \phi) = h\). A \textit{\(\phi\)-twisted conjugacy
      class} is an equivalence class of the equivalence relation of being
      \(\phi\)-twisted conjugate.
		\end{dfn}

    Virtually inner automorphisms and twisted conjugacy classes have been used
    before to study variants of the conjugacy problem in virtually free
    groups; for example in \cite{LadraSilva} to study the generalised
    conjugacy problem.

    We start by studying the set of \(\phi\)-cyclic permutations of a given
    element.

    \begin{lem}
			\label{lem:virt-inner-twisted-cyclic-conjugates}
      Let \(F\) be a finite rank free group and \(\phi \in \Aut(F)\) be
      virtually inner. Let \(g \in F\). Then there is a finite set
      \(X \subseteq F\) and an element \(h \in F\) such that
      the set of \(\phi\)-cyclic permutations of \(g\) is
      \[
        \{h^{-n} p h^{n} \mid n \in \mathbb{Z}, p \in X\}.
      \]
		\end{lem}

		\begin{proof}
      Since \(\phi\) is virtually inner, there exists \(k \in \mathbb{Z}_{>
      0}\) such that \(\phi^k = \psi\) for some inner automorphism \(\psi\).
      Then every \(\phi\)-cyclic
      permutation of \(g = x_1 \cdots x_m\), where each \(x_i\) is a
      generator of \(F\), has the form
      \[
      v (x_{i + 1} \phi^n) (x_{i + 2} \phi^n) \cdots (x_m \phi^n)
      (x_1 \phi^{n + 1}) \cdots (x_{i - 1} \phi^{n + 1}) (u\phi),
      \]
      for some \(n \in \mathbb{Z}\), where \(u, v \in \Sigma^\ast\) are such
      that \(uv = (x_i)\phi^n\). In addition, as \(\phi^k = \psi\), we can
      rewrite this as
      \[
      v (x_{i + 1} \phi^{(n \modns k)} \psi^{\left \lfloor \frac{n}{k} \right \rfloor})
      \cdots (x_m \phi^{(n \modns k)} \psi^{\left \lfloor \frac{n}{k} \right \rfloor})
      (x_1 \phi^{(n \modns k) + 1} \psi^{\left \lfloor \frac{n}{k} \right \rfloor}) \cdots (x_{i - 1} \phi^{(n \modns k) + 1} \psi^{\left \lfloor \frac{n}{k} \right \rfloor}) (u\phi),
      \]
      %
      %
      where \(uv = (x_i)\phi^{(n \modns k)} \psi^{\left \lfloor \frac{n}{k} \right \rfloor}\). Since \(\psi\) is inner, it is defined by conjugation by some element \(h
      \in F\). Let \(r = \left \lfloor \frac{n}{k} \right \rfloor\). Then the above expression becomes
      \begin{equation}
      \label{eqn:virt-inner}
      v h^{-r}(x_{i + 1}) \phi^{(n \modns k)}
      \cdots (x_m \phi^{(n \modns k)})
      (x_1 \phi^{(n \modns k) + 1}) \cdots (x_{i - 1} \phi^{(n \modns k) + 1}) h^r
      (u\phi),
      \end{equation}
      where \(uv = h^{-r}(x_i)\phi^{(n \modns k)}h^r\). We have
      \(u = h^{-r} u' h^r\) and \(v = h^{-r} v' h^r\), where \(u'\) and
      \(v'\) are freely reduced, and \(h^{-r}u'v' h^r
      = uv =  h^{-r} (x_i)\phi^{n \modns k} h^r\), and so \(u'v' =
      (x_i) \phi^{n \modns k}\). Since \(\psi\) is a power of \(\phi\), they
      commute, and so \(u\phi = (h^{-r} u' h^r )\phi = h^{-r} (u' \phi) h^r\).
      We can therefore rewrite \eqref{eqn:virt-inner} as
      \[
      h^{-r} v'(x_{i + 1}) \phi^{(n \modns k)}
      \cdots (x_m \phi^{(n \modns k)})
      (x_1 \phi^{(n \modns k) + 1}) \cdots (x_{i - 1} \phi^{(n \modns k) + 1})
      (u'\phi) h^r,
      \]
      Thus every \(\phi\)-cyclic permutation of \(g\) is an \(h^{r}\)-conjugate of an expression of the form
      \[
        v'(x_{i + 1}) \phi^{(n \modns k)}
        \cdots (x_m \phi^{(n \modns k)})
        (x_1 \phi^{(n \modns k) + 1}) \cdots (x_{i - 1} \phi^{(n \modns k) + 1})
        (u'\phi).
      \]
      Since there are finitely many possibilities for \(i\), \(u'\), \(v'\) and
      \(n \modns k\), these expressions define finitely many elements of \(F\).
      Moreover, all such expressions represent \(\phi\)-cyclic permutations of
      \(g\), and so the set of \(\phi\)-cyclic permutations of \(g\) is in the
      stated form.
		\end{proof}

    Before we can show that the set of \(\phi\)-cyclic permutations of a given
    element forms a recognisably context-free set, we need the following lemma.

    \begin{lem}[{\cite[Lemma 4.6]{one_counter_groups}}]
      \label{lem:free-gp-freely-red-rcf}
      Let \(F\) be a free group with free basis \(\Sigma\) and natural
      homomorphism \(\pi \colon (\Sigma \cup \Sigma^{-1})^\ast \to F\). Let
      \(E \subseteq F\) be such that there is a context-free language \(L
      \subseteq E \pi^{-1}\) that contains the freely reduced forms of every
      element of \(E\). Then \(E\) is recognisably context-free.
    \end{lem}

    We use Lemma~\ref{lem:virt-inner-twisted-cyclic-conjugates} and
    Lemma~\ref{lem:free-gp-freely-red-rcf} to show the following.

    \begin{lem}
      \label{lem:virt-inner-cyclic-perm-rcf}
      Let \(F\) be a finite rank free group and \(\phi \in \Aut(F)\) be
      virtually inner. Let \(\Sigma\) be a free basis for \(F\) and
      let \(w \in (\Sigma \cup \Sigma^{-1})^\ast\) be freely reduced. Let
      \(E \subseteq F\) denote the set of all elements that can be written
      as \(\phi\)-cyclic permutations of \(w\). Then \(E\) is recognisably
      context-free.
    \end{lem}

    \begin{proof}
      Lemma~\ref{lem:virt-inner-twisted-cyclic-conjugates} tells us that there
      is a finite set \(X \subseteq F\) and an element \(h \in F\) such that
      \(E = \{h^{-n} p h^n \mid n \in \mathbb{Z}, p \in X\}\). Since finite
      unions of context-free languages are context-free
      (Lemma~\ref{lem:cf-closure}), it suffices to show that if \(p \in F\),
      then \(E_p = \{h^{-n} p h^n \mid n \in \mathbb{Z}\}\) is recognisably
      context-free. Using Lemma~\ref{lem:free-gp-freely-red-rcf}, it suffices
      to show that there exists a context-free language
      \(L \subseteq E_p \pi^{-1}\) that contains all freely reduced words
      in \(E_p \pi^{-1}\). We construct a context-free grammar for \(L\).

      Let \(u\) denote the freely reduced form \(h\).  Note that there is a
      finite set \(Y\) of words such that the freely reduced words in \(E_p
      \pi^{-1}\) are the freely reduced words in the set
      \[
        \{u^{-n} v u^n \mid n\in \mathbb{Z}, v \in Y\}.
      \]
      Thus if we take \(L = \{u^{-n} v u^n \mid n \in \mathbb{Z}, u \in Y\}\),
      then \(L \subseteq E_p \pi^{-1}\) and \(L\) contains all freely reduced
      words in \(E_p \pi^{-1}\), and so it suffices to show that \(L\) is
      context-free.  Again, since finite unions of context-free languages are
      context-free (Lemma~\ref{lem:cf-closure}), it suffices to show that for
      any freely reduced \(v \in Y\), \(L_v = \{u^{-n} v u^n \mid
      n \in \mathbb{Z}\}\) is context-free.

      Fix \(v \in Y\). We define a context-free grammar for \(L_v\). Our set of
      non-terminals will be \(\{\mathbf S, \mathbf T, \mathbf U\}\), with
      \(\mathbf{S}\) the start symbol. The set of productions \(\mathcal P\) is
      defined by
      \[
        \mathcal P = \{\mathbf{S} \to \mathbf T, \mathbf S \to \mathbf U,
          \mathbf T \to u^{-1} \mathbf T u, \mathbf T \to v,
        \mathbf U \to u \mathbf U u^{-1}, \mathbf U \to v\}.
      \]
      Any derivation using these productions and starting at \(\mathbf S\),
      either goes straight to \(\mathbf T\) or to \(\mathbf U\). When in
      \(\mathbf T\), we can add \(u^{-1}\) at the beginning and \(u\) at the
      end, or replace the \(\mathbf T\) with \(v\). Thus the set of words
      derived with first production \(\mathbf S \to \mathbf T\) is \(\{u^{-n} v
        u^n \mid n \in \mathbb{Z}_{\geq 0}\}\). By symmetry, those derived
        through \(\mathbf U\) are \(\{u^{-n} v u^n \mid n \in \mathbb{Z}_{\leq
        0}\}\), and so the grammar \((\Sigma \cup \Sigma^{-1}, \{\mathbf S,
        \mathbf T, \mathbf U\},
        \mathcal P, \mathbf{S})\) generates \(L_v\), as required.
    \end{proof}

    We now use the fact that the set of \(\phi\)-cyclic permutations of a given
    element is recognisably context-free to prove that the \(\phi\)-twisted
    conjugacy classes in free groups are recognisably context-free in finite
    extensions corresponding to the automorphism \(\phi\).

    The conjugacy problem in free groups can be solved using two facts: every
    conjugacy class has only finitely many cyclically reduced words, and every
    freely reduced word representing an element of a conjugacy class can be
    expressed in the form \(u x u^{-1}\), with \(u\) a freely reduced word and
    \(x\) a cyclically reduced word. The first fact we replace with
    Lemma~\ref{lem:elt-hash-aut}, and the following result is an analogue of
    the latter for \(\phi\)-twisted conjugates.

    \begin{proposition}
      \label{prop:free-twisted-conj-red-form}
      Let \(F\) be a finite rank free group with basis \(\Sigma\), and let
      \(\phi \in \Aut(F)\). Let \(C\) be a \(\phi\)-twisted
      conjugacy class of \(F\). Let \(X\) be the set of freely reduced words
      representing the minimal length elements in \(C\) and their
      \(\phi\)-cyclic permutations. Let \(C_\text{red}\) be the set of freely
      reduced words over \(\Sigma \cup \Sigma^{-1}\) representing elements of
      \(C\). Then
      \begin{equation}
        C_\text{red} \subseteq \{v_1 u v_2 \mid u \in X, v_1, v_2 \text{ are freely
          reduced representatives for } g^{-1}, g\phi, \text{ for some } g
        \in F\}.
         \label{eqn:conj_red}
      \end{equation}
    \end{proposition}

    \begin{proof}
      Note that in the expression \eqref{eqn:conj_red}, we mean equality as
      words; there is no free reduction involved. Let \(w \in C_\text{red}\).
      Then \(w\) is \(\phi\)-twisted conjugate to some element \(h \in F\) such
      that the freely reduced representative for \(h\) lies in \(X\). Thus \(w\)
      can be obtained from \(v_1 u v_2\) by freely reducing, where \(u \in X\)
      is of minimal length, and \(v_1\) and \(v_2\) are the freely reduced
      representatives for \(g^{-1}\) and \(g \phi\), respectively for some \(g
      \in F\). If there is no free reduction between \(v_1\) and \(u\), and
      \(u\) and \(v_2\), we will have shown \eqref{eqn:conj_red}. So we will
      modify the expression to show that such a form will always exist.

      We have \(u \equiv x u' y\), \(v_1 \equiv v_1' x^{-1}\), \(v_2 \equiv y^{-1}
      v_2'\) and \(w \equiv v_1' u v_2'\), for some freely reduced \(x, \ y, \ u', \ v_1', \ v_2'
      \in (\Sigma \cup \Sigma^{-1})^\ast\). If the word \(w
      \equiv v_1' u v_2'\) is of the form \(z^{-1} p (z \phi)\), for freely
      reduced words \(p, \ z, \ z \phi\), and \(z \neq \varepsilon\) then it
      suffices to show that \(p\) lies in the form stated in
      \eqref{eqn:conj_red}. Since \(|p| < |w|\) and \(p \in C_\text{red}\), we
      can use induction to conclude that \(p\) is of the form stated in
      \eqref{eqn:conj_red}.

      Thus we can assume that \(w\) is not of the form \(z^{-1} p (z \phi)\).
      We have that (after freely reducing \(v_1^{-1} \phi\)), that \(v_1^{-1}
      \phi \equiv v_2\), and so \((v_1')^{-1} \phi\) (after being freely
      reduced) and \(v_2'\) are both suffixes of \(v_2\). However, since we are
      assuming that \(w\) is not of the form \(z^{-1} p (z \phi)\), for freely
      reduced words \(p, \ z, \ z \phi\), we have that \(v_2'\) and (the freely
      reduced form of) \((v_1')^{-1} \phi\) must not have a common suffix.
      Since they are both suffixes of the same word, we conclude that one must
      be empty. If \((v_1')^{-1} \phi\) is empty, then \(v_1'\) is empty, since
      \(\phi\) is an automorphism of \(F\). We can therefore split into the cases
      when \(v_1'\) or \(v_2'\) are empty.

			Case 1: \(v_1' = \varepsilon\). \\
			Then \(w \equiv u' v_2'\), and \(v_1 u v_2 \equiv x^{-1} x u' y y^{-1} v_2'\).
			Since \(x \phi \equiv v_1^{-1} \phi =_F v_2 \equiv y^{-1} v_2'\), we can
			\(\phi\)-cyclically permute \(u \equiv x u' y\) to \(u' y y^{-1} v_2'
			=_F u' v_2' \equiv w\). Thus \(w \pi\) is a \(\phi\)-cyclic permutation of
			\(u \pi\), and \(u \in C_\text{red}\), as required.

			Case 2: \(v_2' = \varepsilon\). \\
			Then \(w \equiv v_1' u'\) and \(v_1 u v_2 \equiv v_1' x^{-1} x u' y
			y^{-1}\). Since \(v_1^{-1} \phi =_F v_2 \equiv y^{-1}\), we have that \(y
			\phi^{-1} =_F v_1 \equiv v_1' x^{-1}\). Thus \(u \equiv x u' y\) is a
			\(\phi\)-cyclic permutation (using the relation `backwards') of \((y
			\phi^{-1}) x u' =_F v_1' x^{-1} x u' =_F v_1' u' \equiv w\). Thus \(w
			\pi\) is a \(\phi\)-cyclic permutation of \(u \pi\), and \(u \in
			C_\text{red}\), as required.
    \end{proof}

    \begin{lem}
	    \label{lem:elt-hash-aut}
      Let \(G\) be a finitely generated virtually free group. Let \(F\) be a
      finite-index normal free subgroup of \(G\), and \(T\) be a right
      transversal for \(F\) in \(G\). Let \(\Sigma\) be a basis for \(F\) and
      let \(\# \notin \Sigma \cup \Sigma^{-1} \cup T\) be a new letter. Let
      \(\phi \in \Aut(F)\). Then the language
      \[
				\{u \# v \mid u, \ v \in (\Sigma \cup \Sigma^{-1} \cup T)^\ast, \
				u^{-1}\phi = v\}
			\]
			is context-free.
		\end{lem}

		\begin{proof}
      We construct a pushdown automaton accepting the language \(L_\phi\) stated
      in the lemma. We start by constructing a pushdown automaton \(\mathcal A\)
      that accepts \(g \pi ^ {-1}\) for some \(g \in G\). We first describe the
      automaton \(\mathcal A\).  Our description is based on the proof of
      the Muller-Schupp Theorem in
      \cite{groups_langs_aut}, and generalises Example~\ref{ex:pda}. The idea is
      given a word, we track the normal form \(ht\) of the prefix `read so far',
      by storing the freely reduced form on the stack, and having a state
      \(q_t\) for each \(t \in T\). Let \(t_0\) be the unique element of \(T
      \cap F\). Thus our stack alphabet will be \(\{\perp\} \cup \Sigma \cup
      \Sigma^{-1}\), where\ \(\perp\) is the bottom of stack symbol. We add an
      additional state \(p\) to be our unique accept state, and have a
      transition from \(q_{t_h}\) with stack \(w\) to \(p\), where \(t_g \in T\)
      is the unique coset representative such that \(g = h_g t_g\) for some
      (unique) \(h_g \in F\), and \(w\) is the freely reduced form of \(h_g\).
      Our start state will simply be \(q_{t_0}\).

      It remains to add transitions between the states \(q_t\) for each \(t \in
      T\). We will be using \(\Sigma \cup \Sigma ^ {-1} \cup \{\chi\}\) as our
      stack alphabet, with \(\chi\) the bottom of stack symbol. For each \(a
      \in \Sigma \cup \Sigma ^ {-1} \cup T\) and each \(t \in T\), we have that
      \(ta = ht'\) for some \(h \in F\) and \(t' \in T\). Thus when in state
      \(q_t\) and reading \(a\), we want to transition to state \(q_{t'}\),
      then add the freely reduced form \(\mu_h\) of \(h\) to the stack, and
      then freely reduce. Since we cannot simply freely reduce the stack, we
      have multiple transitions from \(q_t\) to \(q_{t'}\) when reading \(a\);
      one transition for each pair \((\omega, \ x)\), where \(\omega\) is a
      suffix of the freely reduced form of \(h ^ {-1}\) (that is,
      \(\omega^{-1}\) is a prefix of \(\mu_h\)) and \(x \in \Sigma \cup \Sigma
      ^ {-1}\) is such that \(x \omega\) is not a suffix of \(\mu_h^{-1}\).  This
      transition then pops \(x \omega\) from the stack, and pushes \(x\)
      followed by the `remainder' of \(\mu_h\); that is the (unique) freely
      reduced word \(\nu\), such that \(\mu_h \equiv \omega ^{-1} \nu\). To
      deal with the case when the stack is empty (that is, it contains only the
      symbol \(\perp\), we add the transition from \(q_t\) to \(q_{t'}\) when
      reading \(a\) that pops \(\perp\) and pushes \(\perp \mu_h\).

      We now construct a pushdown automaton accepting \(L_\phi\). We start by
      taking two (disjoint) copies of the automaton \(\mathcal{A}\): \(\mathcal
      A_1\) and \(\mathcal A_2\). The automaton \(\mathcal A_1\) will be for \(g
      \pi^{-1}\) for some fixed \(g \in G\) with its accept state considered not
      an accept state (and so the choice of \(g\) does not matter), and
      \(\mathcal A_2\) will be for \(1 \pi^{-1}\). We modify the transitions of
      \(\mathcal A_1\), so that whenever we read \(a \in \Sigma \cup \Sigma ^
      {-1} \cup T\), we instead use the transition for \(a^{-1}\). That is, if
      we were in state \(t\) with stack \(\mu_h\), we would end up in the state
      and stack corresponding to \(\mu_h t a^{-1}\). Similarly, we modify
      \(\mathcal A_2\) so that whenever we read \(a\) we act as if we read \(a
      \phi\). Again, this means that in state \(t\) with stack \(\mu_h\), we
      move to the state-stack pair corresponding to \(\mu_h t (a \phi)\). Our
      start state will be the start state of \(\mathcal A_1\), and our accept
      state will be the accept state of \(\mathcal A_2\).

      We add an additional transition between every state in \(\mathcal A_1\)
      to the corresponding state in \(\mathcal A_2\) that does not alter the
      stack (that is, it pushes and pops \(\varepsilon\) from the stack) when
      reading \(\#\).
		\end{proof}

    We can now show that \(\phi\)-twisted conjugacy classes of free groups are
    recognisably context-free in finite-index overgroups.

    \begin{proposition}
      \label{prop:twisted-conj-cf}
      Let \(F\) be a finite rank free group and \(\phi \in \Aut(F)\) be
      virtually inner. Let \(G\) be such that \(F\) is a finite-index normal
      subgroup of \(G\). Then every \(\phi\)-twisted conjugacy class of \(F\)
      is recognisably context-free in \(G\).
		\end{proposition}

    \begin{proof}
      Let \(\Sigma\) be a free basis for \(F\) and \(\pi \colon (\Sigma \cup
      \Sigma^{-1})\ast \to F\) be the natural map. Let \(C_\phi\) be a
      \(\phi\)-twisted conjugacy class.  Using
      Lemma~\ref{lem:free-gp-freely-red-rcf}, in order to show that \(C_\phi\)
      is recognisably context-free, it suffices to construct a context-free
      language \(L \subseteq C_\phi \pi^{-1}\) that contains every freely
      reduced word in \(C_\phi \pi^{-1}\).
      Proposition~\ref{prop:free-twisted-conj-red-form}, shows us that all such
      words are of the form \(v_1 u v_2\), where \(v_1\) and \(v_2\) are freely
      reduced representative for \(g^{-1}\) and \(g \phi\), for some \(g \in
      F\), and \(u\) is a \(\phi\)-cyclic permutation of a minimal length word
      in \(C_\phi\).

      It therefore suffices to show that the language of all words of the form
      \(v_1 u v_2\) where \(v_1\) and \(v_2\) are (not necessarily freely
      reduced) representatives for \(g^{-1}\) and \(g \phi\), for some \(g \in
      F\), and \(u\) is a \(\phi\)-cyclic permutation of a minimal length word
      in \(C_\phi\), is a context-free language. Lemma~\ref{lem:elt-hash-aut}
      tells us that
      \[
				\{u \# v \mid u, \ v \in (\Sigma \cup \Sigma^{-1} \cup T)^\ast, \
				u^{-1}\phi = v\}
			\]
      is context-free. In addition, if \(E \subseteq F\) denotes the set of
      elements that can be expressed as \(\phi\)-cyclic permutations of a
      minimal length word in \(C_\phi\), then
      Lemma~\ref{lem:virt-inner-cyclic-perm-rcf}, together with the fact that
      finite unions of context-free languages are context-free
      (Lemma~\ref{lem:cf-closure}) shows that \(E \pi^{-1}\) is context-free.
      Since context-free languages are closed under substitutions by other
      context-free languages (Lemma~\ref{lem:cf-subs}), the language
      \[
				L = \{u x v \mid u, \ v \in (\Sigma \cup \Sigma^{-1} \cup T)^\ast, \
        u^{-1}\phi = v, x \in E \pi^{-1}\}
			\]
      is context-free. As \(L \subseteq C_g \pi^{-1}\), and by construction
      contains every freely reduced word in \(C_g \pi^{-1}\),
      Lemma~\ref{lem:free-gp-freely-red-rcf} tells us that \(C_g\) is
      recognisably context-free.
    \end{proof}

    We can now show that conjugacy classes in finitely generated virtually free
    groups are recognisably context-free.

	  \begin{proposition}
			\label{prop:VF-conj-forward}
	    Let \(G\) be a finitely generated virtually free group. Then every conjugacy
	    class of \(G\) is recognisably context-free.
	  \end{proposition}

		\begin{proof}
      Let \(F\) be a finite-index normal free subgroup of \(G\). Fix a finite
      right transversal \(T\) for \(F\) in \(G\). We have that every element of
      \(G\) can be written in the form \(ht\) where \(h \in F\) and \(t \in
      T\). Since \(F\) is finitely generated free, we can write every element
      of \(F\) uniquely as a freely reduced word with respect to a (finite)
      basis \(\Sigma\). We will use \(\Sigma \cup \Sigma^{-1} \cup T\) as our
      (monoid) generating set. Since \(F\) is normal and finite-index, elements
      of \(G\) (in particular elements of \(T\)) act on \(F\) by automorphisms
      of finite order. For each \(t \in T\) we write \(\phi_t\) to denote the
      automorphism of \(F\) defined by \(h \mapsto tht^{-1}\). Note that
      for each \(t \in T\) there exists \(k \in \mathbb{Z}_{> 0}\) such
      that \(t^k \in F\), and so \(\phi_t^k \colon x \mapsto t^{-k} xt^k\),
      and we have shown that \(\phi_t^k\) is an inner automorphism of
      \(F\). Thus all automorphisms \(\phi_t\) are virtually inner.

      Fix a conjugacy class \(C\) of \(G\), and a representative \(h_0 t_0 \in
      C\).  Let \(ht \in G\). We have that \(ht \in C\) if and only if there
      exists \(x \in G\) such that \(x h_0 t_0 x^{-1} = h t\). We can write any
      such \(x = y s\), where \(y \in F\) and \(s \in T\). So
      \begin{equation}
        \label{eqn:virt-free-cong}
				h t = ys ht_0 s^{-1} y^{-1} = y (h_0 \phi_s) (y^{-1} \phi_s^{-1}
			  \phi_{t_0} \phi_s) st_0 s^{-1}.
      \end{equation}
      Thus \(h t \in C\) if and only if there exists \(y \in F\) and \(s \in
      T\) such that \eqref{eqn:virt-free-cong} is satisfied. Moreover, if we
      fix \(s \in T\), then \eqref{eqn:virt-free-cong} becomes a twisted
      conjugacy class of \(F\), using the (fixed) automorphism \(\phi_s\),
      multiplied by a fixed element of \(F\) and then a fixed element of \(T\);
      that is, the normal form for the fixed element \(s t_0 s ^ {-1}\). Since
      finite unions of context-free languages are context-free
      (Lemma~\ref{lem:cf-closure}), it is sufficient to show that the set of
      elements that lie in a set of the form \(C_\phi h s\) is recognisably
      context-free, where \(C_\phi\) is a \(\phi\)-twisted conjugacy class of
      \(F\), \(h \in F\) and \(s \in T\). This follows from
      Corollary~\ref{cor:post-multiply-cf} together with the fact that
      \(C_\phi\) is recognisably context-free
      (Proposition~\ref{prop:twisted-conj-cf}).
    \end{proof}

		Combining Proposition~\ref{prop:VF-conj-forward} with the fact that
		\(\{1\}\) is a conjugacy class that is recognisably context-free in a
		group \(G\) if and only if \(G\) is virtually free (by the Muller-Schupp
		Theorem), we have the following:

		\begin{theorem}
			\label{thm:conj}
			Let \(G\) be a finitely generated group. Then every conjugacy class of
			\(G\) is recognisably context-free if and only if \(G\) is virtually
			free.
		\end{theorem}

    The following example, due to Corentin Bodart, shows that there exist
    non-virtually free groups that admit recognisably deterministic
    context-free conjugacy classes (they are in fact recognisable).

    \begin{ex}
      Let \(H\) be a finitely generated abelian group and let
      \[
        G = \langle H \cup \{t\} \mid \{tht = h^{-1} \mid h \in H\}
        \cup \{t^2 = 1\}\rangle.
      \]
      We look at the conjugacy class of \(t\) in \(G\). Since
      \(th = h^{-1} t\) for all \(h \in H\), we have for all \(h \in H\) that
      \(h^{-1} th = h^{-2}t\). In particular, the conjugacy class of
      \(t\) contains the coset \(H^2 t\). Conversely, \(if h^2 t \in H^2 t\),
      then \(h^2 t = hth^{-1}\), and so \(h^2t\) is conjugate to \(t\). We
      can thus conclude that the conjugacy class of \(t\) is equal to
      the coset \(H^2 t\).

      In addition, \(H\) is index \(2\) in \(G\) and \(H^2\) is finite index in
      \(H\), and so \(H^2\) is finite-index in \(G\). By
      Proposition~\ref{prop:rec-classification}, \(H^2 t\) is
      recognisable in \(G\), and so the conjugacy class of \(t\) is
      recognisable in \(G\).
    \end{ex}

\section{Subgroups and cosets}
	\label{sec:cosets}
  We conclude by giving a classification of when subgroups and cosets with
  quasi-transitive Schreier coset graphs of finitely generated groups are
  recognisably context-free. Ceccherini-Silberstein and Woess
  provided a full classification by showing that a subgroup (and hence coset,
  using Corollary~\ref{cor:post-multiply-cf}) is recognisably context-free if
  and only if the Schreier coset graph is what is called a context-free graph
  \cite{CeccWoess2012}, a term introduced by Muller and Schupp
  \cite{MullerSchupp85} which depends on the ends of a graph. Woess
  continued the study of these graphs in \cite{Woess2012}.

  As mentioned earlier, since the release of this work, the author has been made
  aware of a result of Rodaro, released a few months earlier that proves the
  main result of this section \cite{rodaro}, when taken together with the result
  of Ceccherini-Silberstein and Woess \cite{CeccWoess2012} that classifies when
  a Schreier coset graph is a context-free graph. Rodaro's method uses the
  context-free graphs introduced by Muller and Schupp \cite{MullerSchupp85}. We
  prove this using a recent generalisation of Stallings' Theorem
  \cite{HamannLehnerMiraftabRuhmann}, avoiding the notion of a context-free
  graph.

  We consider the case when a Schreier coset graph is quasi-transitive; that
  is, it has finitely many automorphic orbits, and show that a coset with
  quasi-transitive coset graph is recognisably context-free if and only if the
  corresponding Schreier coset graph is a quasi-tree. If the subgroup in
  question were normal, we could use the Muller-Schupp Theorem to show that the
  Cayley graph of the quotient group must be a quasi-tree.  Since this is
  always isomorphic to the Schreier coset graph, this proves the result. Stated
  in terms of properties of the quotient rather than Schreier coset graphs this
  is a coset \(Hg\) within a group \(G\) is recognisably context-free if and
  only if \(G / H\) is virtually free.

  Extending this to the non-normal case requires more work. In light of
  Corollary~\ref{cor:post-multiply-cf}, it is sufficient to answer the question
  for subgroups. The fact that a subgroup \(H\) being recognisably context-free
  implies that the Schreier coset graph of \(H\) is a quasi-tree is not too
  difficult to show using the same argument as the Muller-Schupp Theorem. The
  converse of this is much more difficult. The main stumbling block arises from
  the fact that the Muller-Schupp proof shows that groups with context-free
  word problem have Cayley graphs quasi-isometric to trees, then uses
  Stallings' Theorem and Dunwoody's accessibility result to show that these
  groups must be virtually free, and then shows that virtually free groups have
  context-free word problem. The difficulty we have here is replacing
  Stallings' Theorem and Dunwoody's accessibility result, as we are working
  with Schreier coset graphs rather than groups.

  A recent result of Hamann, Lehner, Miraftab and R\"{u}hmann does prove a
  version of Stallings' Theorem for connected quasi-transitive graphs
  \cite{HamannLehnerMiraftabRuhmann}. In
  addition, they show that such quasi-trees will be `accessible' in their
  sense. We state these results in Subsection~\ref{subsec:tree-amalg}.

  We begin with the more straightforward direction.

	\begin{proposition}
    \label{prop:Schreier-coset-qt}
    Let \(G\) be a finitely generated group and \(H \leq G\) be such that the
    Schreier coset graph of \(H\) is quasi-transitive. If \(H\) is recognisably
    context-free then the Schreier coset graph of \(H\) in \(G\) is a
    quasi-tree.
	\end{proposition}

	\begin{proof}
    Fix a finite monoid generating set \(\Sigma\) for \(G\) and let \(\pi
    \colon \Sigma^\ast \to G\) be the natural homomorphism. As \(H \pi^{-1}\)
    is context-free, Lemma~\ref{Chomsky_norm_form_lem} tells us there is a
    context-free grammar \(\mathcal G = (V, \ \Sigma, \ \mathcal P, \
    \mathbf{S})\) that is in Chomsky normal form and has no useless
    non-terminals, such that the language of \(\mathcal G\) is \(H \pi^{-1}\).
    Let \(\mathbf{A} \in V\), and suppose \(w_1, \ w_2 \in \Sigma^\ast\) are
    such that \(\mathbf{A} \Rightarrow^\ast w_1\) and \(\mathbf{A}
    \Rightarrow^\ast w_2\). Then there exist \(\sigma, \ \tau \in \Sigma^\ast\)
    such that \((\sigma w_1 \tau)\pi, \ (\sigma w_2 \tau) \pi \in H\). So
    \((\sigma w_1 w_2^{-1} \sigma^{-1}) \pi = (\sigma w_1 \tau \tau^{-1}
    w_2^{-1} \sigma^{-1}) \pi \in H\). We have thus shown that \((w_1
    w_2^{-1})\pi\) is conjugate to an element of \(H\). Moreover, a conjugating
    element is \(\sigma \pi\).

    For each \(\mathbf{A} \in V\) choose a word \(w_\mathbf{A} \in
    \Sigma^\ast\) such that \(\mathbf{A} \Rightarrow^\ast w_\mathbf{A}\) (such
    a derivation always exists as \(\mathbf{A}\) is not useless). Let \(M =
    \max\{|w_\mathbf{A}| \mid \mathbf{A} \in V\}\).  Suppose \(u \in H
    \pi^{-1}\). Then \(w\) labels a circuit in the Schreier coset graph
    \(\Gamma\) of \(H\), with basepoint \(H\). Suppose in a derivation of \(u\)
    in \(\mathcal G\) we have \(\mathbf{S} \Rightarrow^\ast \sigma \mathbf{A}
    \tau \Rightarrow^\ast \sigma v \tau \equiv w\). Note that \((v^{-1}
    w_\mathbf{A})\pi \in H^{\sigma \pi}\). In particular, \((v^{-1}
    w_\mathbf{A}) \pi\) lies in the stabiliser of \(H (\sigma \pi)\). So
    replacing \(\sigma v v^{-1} w_\mathbf{A} \tau\) also traces a circuit in
    \(\Gamma\) with basepoint \(H\), and so \(\sigma w_\mathbf{A}
    \tau\) does as well. Since \(|w_\mathbf{A}| \leq M\), this will be an
    \(M\)-reduction. We can therefore apply the Muller-Schupp method to
    \(M\)-triangulate every circuit in \(\Gamma\) with basepoint
    \(H\) by replacing subwords derived from each non-terminal \(\mathbf{A}\)
    with \(w_{\mathbf{A}}\).

    To achieve this, we go through the derivation of a word \(u\) in \(H
    \pi^{-1}\), (ignoring productions of the form \(\mathbf{A} \to a\)), and
    for each production of the form \(\mathbf{A} \to \mathbf{BC}\), we have an
    \(M\)-reduction from the start of the subword of \(u\) derived from
    \(\mathbf{BC}\) to the end (using the label \(w_\mathbf{A}\)).

    To show that every circuit in \(\Gamma\) is triangulable, not just those
    with a basepoint in \(H\), it is sufficient to show that for all
    automorphic orbits of the Schreier coset graph there is a basepoint \(Ht\)
    such that every circuit with a basepoint \(Ht\) is triangulable. Fix a set
    \(T\) of representatives for these automorphic orbits. Since
    \(\Gamma\) is quasi-transitive, we can choose \(T\) to be finite.

    Let \(Ht \in T\) and let \(u \in \Sigma^\ast\) trace a path in \(\Gamma\)
    from \(Ht\) to \(Ht\). Fix a word \(w_t \in \Sigma^\ast\) representing
    \(t\). Then \(w_t u w_t^{-1}\) labels a path in \(\Gamma\) from \(H\)
    to itself, and so this circuit is \(M\)-triangulable. Thus the circuit
    traced by \(u\) with basepoint \(Ht\) is \((M + |w_t|)\)-triangulable.
    Let \(K = \max_{Ht \in T} |w_t|\). We can conclude that the Schreier graph
    of \(H\) in \(G\) is \((M + K)\)-triangulable, and so by
    Theorem~\ref{thm:qt-triangulation}, it is a quasi-tree.
	\end{proof}

  We now prove that if a Schreier coset graph of a subgroup of a finitely
  generated group is a quasi-tree, then then the subgroup is recognisably
  context-free. We first need some definitions, based on definitions in
  \cite{MullerSchupp85}.

	\begin{dfn}
		A \textit{finitely generated graph} is a \(\Sigma\)-labelled graph,
		where \(\Sigma\) is an alphabet, such that
		\begin{enumerate}
			\item \(\Gamma\) is connected;
			\item \(\Gamma\) has uniformly bounded degree (that is, there exists
			\(d > 0\) such that the degree of every vertex is at most \(d\));
			\item \(\Sigma\) is finite.
		\end{enumerate}
    Let \(\Gamma\) be a finitely generated graph, with edges labelled using an
    alphabet \(\Sigma\). Fix a vertex \(v_0\) and a finite set \(F\) of
    vertices of \(\Gamma\). The \textit{language} of \(\Gamma\) with respect to
    the \textit{origin} \(v_0\) and \textit{accepting states} \(F\) is the set
    of all words that trace a path in \(\Gamma\) from \(v_0\) to a vertex in
    \(F\).
	\end{dfn}

  We now show that taking a `nice' tree amalgamation of finitely generated
  graphs that both have context-free languages yields a graph with a context-free
  language.

	\begin{lem}
		\label{lem:tree-amalg-cf}
		Let \(\Gamma_1\) and \(\Gamma_2\) be finitely generated quasi-transitive
		graphs with edges labelled from an alphabet \(\Sigma\), whose languages are
		context-free with respect to any origin and any finite set of accepting
		states. Then every tree amalgamation of \(\Gamma_1\) and \(\Gamma_2\) with
		finite adhesion and finite identification also has a context-free
		language with respect to any origin and any finite set of accepting states.
	\end{lem}

	\begin{proof}
    Let \((S_k^i)_{k \in I_i}\) be the adhesion sets of the tree amalgamation
    in \(\Gamma_i\), for each \(i \in \{1, 2\}\), and assume \(I_1\) and
    \(I_2\) are disjoint. Let \(T\) be the connecting tree. Note that as finite
    unions of context-free languages are context-free
    (Lemma~\ref{lem:cf-closure}), it suffices to show that the language of
    \(\Gamma_1 \ast \Gamma_2\) is context-free with respect to any origin and
    any singleton set of accepting states.

		Fix an origin vertex \(u_0\) in \(\Gamma_1 \ast \Gamma_2\). This lies in a
		copy of \(\Gamma_1\) or \(\Gamma_2\); without loss of generality assume it
		lies in a copy of \(\Gamma_1\), corresponding to a pair of vertices \((v_0,
		t_0) \in V(\Gamma_1) \times V(T)\). Since \(\Gamma_1\) has a context-free
		language with respect to any origin and any finite set of accepting states,
		for each \(u \in \bigcup_{k \in I_1} S_k^1\), there is a pushdown automaton
		that accepts the language of \(\Gamma_1\) with respect to the origin \(v_0\)
		and the accepting state \(u\).

		We can then take the finite union of these pushdown automata across all \(u
		\in \bigcup_{k \in I_1} S_k^1\), to obtain a pushdown automaton \(\mathcal
		A_0\) that accepts the language of all words that trace a path in
		\(\Gamma_1\) from \(v_0\) to a vertex in \(\bigcup_{k \in I_1} S_k^1\).
		Moreover, as we constructed this as a (disjoint) finite union of pushdown
		automata, we can assume that the set of accept states is partitioned into
		the vertices lying in \(\bigcup_{k \in I_1} S_k^1\).

		Now let \(i \in \{1, 2\}\). Since the language of \(\Gamma_i\) is
		context-free with respect to any origin and any set of accepting states, for
		each \(u, v \in \bigcup_{k \in I_i} S_k^i\), we can construct a pushdown
		automaton \(\mathcal B_{i, u, v}\) accepting the language of all words that
		trace a path in \(\Gamma_i\) from \(u\) to \(v\).

    Fix an accepting state \(q \in V(\Gamma_1 \ast \Gamma_2)\).  Recall that
    the directed edge `version' \(D(T)\) of \(T\) admits an edge-labelling
    using \(I_1 \sqcup I_2\). Since \(D(T)\) is the directed `version' of a
    tree, there is a unique minimal path in \(D(T)\) from \(t_0\) to each
    vertex \(t \in V(T)\). This path traces a word \(w_t \in (I_1 \sqcup
    I_2)^\ast\), and thus we can uniquely describe each vertex in \(T\) using a
    word over \(I_1 \sqcup I_2\). Let \(t_q \in V(T)\) and \(v_q \in \Gamma_i\)
    be a (not necessarily unique) pair, such that \(q\) is the image of \(v_q\)
    under the canonical map from the copy of \(\Gamma_1\) or \(\Gamma_2\)
    corresponding to \(t_q\) to \(\Gamma_1 \ast \Gamma_2\). Fix \(j_q \in \{1,
    2\}\) such that \(\Gamma_{j_q}\) is the corresponding graph.

		Similar to the pushdown automata \(\mathcal B_{i, u, v}\) for each vertex
		\(u \in \bigcup_{k \in I_{j_q}} S_k^i\), we construct a pushdown automaton
		\(\mathcal C_{u}\) to be the pushdown automata that accept the language of
		all words that trace a path in \(\Gamma_{j_q}\) from \(u\) to \(v_q\).

		We can assume that all of the pushdown automata we have defined have
		pairwise disjoint sets of states and stack alphabets. We also assume the
		stack alphabets are all pairwise disjoint from \(I_1\) and \(I_2\). By
		Lemma~\ref{lem:pda-empty-stack}, we can assume that whenever a word is
		accepted by any of these pushdown automata, the stack is empty; that is, the
		only symbol on the stack is the bottom of stack symbol.

		We now use the pushdown automata \(\mathcal B_{i, u, v}\), \(\mathcal
		C_{x}\) and \(\mathcal A_0\) to construct a (non-deterministic) pushdown
		automaton \(\mathcal D\) accepting the language of \(\Gamma_1 \ast
		\Gamma_2\) as follows:
		\begin{enumerate}
			\item Our set of states will be
			\[
				Q(\mathcal A_0) \sqcup\bigsqcup_{i \in \{1, 2\}} \bigsqcup_{u, v \in
				\bigcup_{k \in I_i} S_k^i}
				Q(\mathcal B_{i, u, v}) \sqcup \bigsqcup_{u \in \bigcup_{k \in I_{j_q}} S_k^{j_q}}
				Q(\mathcal C_{u}) \sqcup \{p\}
			\]
			where \(p\) is a new state.
			\item Our alphabet will be \(\Sigma\).
			\item Our start state will be the start state \(q_0\) of \(\mathcal A_0\).
			\item Our accept state will be \(p\).
			\item Our stack alphabet will be
			\[
				\chi(\mathcal A_0) \sqcup \bigsqcup_{i \in \{1, 2\}} \bigsqcup_{u, v \in
				\bigcup_{k \in I_i} S_k^i}
				\chi(\mathcal B_{i, u, v})  \sqcup \bigsqcup_{u \in \bigcup_{k \in I_{j_q}} S_k^{j_q}}
				\chi(\mathcal C_{u}) \sqcup  I_1 \sqcup I_2.
			\]
			\item Our bottom of stack symbol will be the bottom of stack symbol
			\(\perp_0\) of \(\mathcal A_0\).
			\item Our transitions will be all of those of the following forms:
			\begin{enumerate}
				\item All transitions entirely within \(\mathcal A_0\) or some
				\(\mathcal B_{i, u, v}\) or \(\mathcal C_u\);
				\item For each \(i \in \{1, 2\}\), each \(u, v \in \bigcup_{k \in I_i}
				S_k^i\), and each bonding map \(\phi\) such that \(v \in \dom \phi\),
				there is an \(\varepsilon\)-transition from each accept state of
				\(\mathcal B_{i, u, v}\) to the start state of \(\mathcal B_{j, v \phi,
				x}\), where \(j \in \{1, 2\} \setminus \{i\}\) and for every \(x \in
				\bigcup_{k \in I_j} S_k^j\). Since all of the automata \(\mathcal A_0\)
				and \(\mathcal B_{i, u, v}\) have empty stacks when arriving in an
				accept state, when making this transition, the stack will have the form
				\(\perp_0 w \perp_{i, u, v}\), where \(w \in (I_1 \sqcup I_2)^\ast\) and
				\(\perp_{i, u, v}\) is the start symbol of \(\mathcal B_{i, u, v}\). We
				pop \(\perp_{i, u, v}\) from the stack, along with the topmost symbol
				\(k\) in \(w\). If \(\dom \phi \neq S_k^i\), we push \(k\) back onto the
				stack, followed by the (unique) \(l \in I_i\) such that \(\im \phi =
				S_l^j\), and then the bottom of stack symbol for \(\mathcal B_{j, v\phi,
				x}\). If \(\dom \phi = S_k^i\), then we don't push \(k\) back onto the
				stack; we only push the bottom of stack symbol for \(\mathcal B_{j,
				v\phi, x}\).
				\item There are transitions analogous to those in (b), except starting
				in \(\mathcal A_0\) and ending in some \(\mathcal B_{i, u, v}\). To
				avoid any ambiguity, we formally state these as well. As mentioned
				earlier, the accept states of \(\mathcal A_0\) are partitioned into
				parts corresponding to the vertices \(u \in \bigcup_{k \in I_1} S_k^1\).
				Fix such a vertex \(u\). For each bonding map \(\phi\) such that \(u \in
				\dom \phi\), there is an \(\varepsilon\)-transition from each accept
				state of \(\mathcal A_0\) that lies in the part of the partition
				corresponding to \(u\) to the start state of \(\mathcal B_{2, v \phi,
				x}\) for all \(x \in \bigcup_{l \in I_2} S_k^l\). In such a case the
				stack will be of the form \(\perp_0\), and we pop \(\perp_0\) from the stack
				then push \(\perp_0 l\), where \(l \in I_2\) is unique such that
				\(\im \phi = S_l^2\).
				\item There are transitions analogous to those in (b), except starting
				in some \(\mathcal B_{i, u, v}\) and ending in some \(\mathcal C_x\).
				Naturally, these only start in automata \(\mathcal B_{i, u, v}\) where
				\(i \neq j_q\), as \(\mathcal C_x\) corresponds to \(\Gamma_{j_q}\).
				\item If \(j_q \neq 1\), then there are transitions analogous to
				those in (c), starting in \(\mathcal A_0\) and ending in some
				\(\mathcal C_u\).
				\item From each accept state of each \(\mathcal C_u\), there is an
				\(\varepsilon\)-transition to \(p\), that pops \(\perp_0 w_{t_q} \$_u\),
				where \(\$_u\) is the bottom of stack symbol of \(\mathcal C_u\). We
				then push \(\perp_0\) back onto the stack.
			\end{enumerate}
		\end{enumerate}
		The automaton \(\mathcal D\) works as follows. The automata \(\mathcal
		A_0\), \(\mathcal B_{i, u, v}\) and \(\mathcal C_x\) simulate the copies of
		\(\Gamma_1\) and \(\Gamma_2\) used to define \(\Gamma_1 \ast \Gamma_2\). We
		use the stack (behind the bottom of stack symbol of whichever automaton we
		are currently in) to track the position within the connecting tree \(T\)
		that we are in, with \(t_0\) being used as a root. The transitions between
		each of the automata \(\mathcal A_0\), \(\mathcal B_{i, u, v}\) and
		\(\mathcal C_x\) simulate the bonding maps, as they identify vertices in
		\(\Gamma_1\) and \(\Gamma_2\). We need \(\mathcal A_0\) to be a separate
		automaton to deal with the multiple accepting vertices we can start with
		(after that, we just pass to a different automaton \(\mathcal B_{i, u, v})\)
		for different accepting states \(v\)). The automaton \(\mathcal C_x\) is
		separate to make transitioning to the accept state \(p\) more
		straightforward. Transitioning from \(\mathcal C_x\) to \(p\) requires us to
		be in the vertex \(v_q\) corresponding to \(q\), and the transition confirms
		that our stack reads \(t_q\); that is, we are in the correct position within
		\(T\), before accepting.
	\end{proof}
  Using Theorem~\ref{thm:qt-fact}, we can build any Schreier coset graph by
  iteratively taking tree amalgamations, starting with a collection of finite
  graphs. Lemma~\ref{lem:tree-amalg-cf} tells us that each of these tree
  amalgamations preserves the property of having a context-free language.  We
  can thus use this to show that a subgroup whose Schreier coset graph is a
  quasi-transitive quasi-tree will be recognisably context-free. We now
  formally state the characterisation of recognisably context-free cosets of
  subgroups with quasi-transitive Schreier coset graphs we have been working
  towards.

	\begin{theorem}
		\label{thm:cosets}
		Let \(G\) be a finitely generated group, \(H \leq G\) and \(g \in G\) be
    such that the Schreier coset graph of \((G, H)\) is quasi-transitive.
		Then \(Hg\) is recognisably context-free if and only if the Schreier coset
		graph of \((G, H)\) is a quasi-tree.
	\end{theorem}

	\begin{proof}
		In light of Corollary~\ref{cor:post-multiply-cf}, it suffices to show
		that \(H\) is recognisably context-free if and only if the Schreier
		coset graph \(\Gamma\) of \((G, H)\) is a quasi-tree. The fact that
		\(H\) is recognisably context-free implies that \(\Gamma\) is a quasi-tree
		is Proposition~\ref{prop:Schreier-coset-qt}. So it remains to show that
		if \(\Gamma\) is a quasi-tree tree, then \(H\) is recognisably context-free.

    Suppose \(\Gamma\) is a quasi-tree. Since \(\Gamma\) is quasi-transitive,
    we can apply Theorem~\ref{thm:qt-fact} to show that \(\Gamma\) can be built
    from a (finite) collection of finite graphs by successive tree amalgamations
    with finite adhesion and finite identification. Each of the finite graphs
    will have a context-free language with respect to any origin and any set of
    accepting states, as the languages of finite graphs are always regular. We
    can then apply Lemma~\ref{lem:tree-amalg-cf} to show that each of the
    successive tree amalgamation preserves the properties of having a
    context-free language with respect to any origin and any set of accepting
    states. Thus \(\Gamma\) has a context-free language with respect to any
    origin and any set of accepting states. In particular, the language of all
    words that trace a path in \(\Gamma\) from \(H\) to \(H\) is context-free.
    Since this is precisely the set of words in \(H \pi^{-1}\), \(H\) is
    recognisably context-free.
	\end{proof}

	\begin{rmk}
		Recall that Herbst and Thomas proved that a subset \(E\) of a group \(G\) is
		recognisable if and only if \(E\) is a finite union of cosets of some
		finite-index subgroup of \(G\) (Proposition~\ref{prop:rec-classification}).
		In light of Theorem~\ref{thm:cosets} (or alternatively the classification of
		when generic subgroups are recognisably context-free due to
		Ceccherini-Silberstein and Woess \cite{CeccWoess2012}) it is natural to ask
		whether an analogous statement may be true for recognisably context-free
		subsets, using subgroups with quasi-tree Schreier coset graphs in place of
		finite-index subgroups (and not necessarily use a fixed subgroup). If we
		consider \(\mathbb{Z}\), then it is easy to see that such a statement cannot
		be true. As the coword problem of \(\mathbb{Z}\) is context-free (by the
		Muller-Schupp Theorem), \(\mathbb{Z}
		\setminus \{0\}\) is a recognisably context-free subset of \(\mathbb{Z}\).
		The only subgroups of \(\mathbb{Z}\) with Schreier coset graphs that are
		quasi-trees are infinite ones. Their cosets are of the form \(\{a x + b\mid
		x \in \mathbb{Z}\}\) with \(a \in \mathbb{Z} \setminus\{0\}\) and \(b \in
		\mathbb{Z}\). It is not difficult to see that any finite union of these sets
		that does not contain zero, must necessarily miss infinitely many elements.
	\end{rmk}

\section*{Acknowledgements}
  I would like to thank Corentin Bodart, André Carvalho, Gemma Crowe, Luke
  Elliott, Matthias Hamann, Mark Kambites, Alan Logan,
  Carl-Fredrik Nyberg Brodda, Davide Perego and N\'{o}ra Szak\'{a}cs
  and Martin van Beek for answering questions, mathematical discussions,
  directing me to references and pointing out errors in previous versions, all
  of which greatly helped with this work. During this work, I was supported by
  the Heilbronn Institute for Mathematical Research.

\bibliography{references}
\bibliographystyle{abbrv}
\end{document}